%% file: paper.tex
\documentclass[oneside]{amsart} 
\usepackage[percent]{overpic}
\usepackage{esint}

\usepackage{paralist}
\usepackage{graphics} 
\usepackage{epsfig} 
\usepackage{xspace}
\usepackage{marginnote} 

\usepackage{array}
\usepackage{threeparttable}
\usepackage{changepage}

\usepackage{appendix}
\usepackage[english]{babel}
\usepackage[T1]{fontenc} 
\usepackage[utf8]{inputenc} 
\usepackage{geometry}
\usepackage{graphicx}
\usepackage{tabto}
\usepackage{amsmath}
\usepackage{calligra}
\usepackage{amsthm}
\usepackage{bbm}
\usepackage{cancel}
\usepackage{mathrsfs}
\usepackage{mathtools}
\usepackage{version}
\usepackage{arydshln}
\usepackage[math]{cellspace}
\usepackage{multicol}
\usepackage{bm}
\usepackage{subfig}

\usepackage{tikz}
\usepackage{amssymb}
\usepackage{algorithm,algpseudocode}

\long\def\NOTE#1{}

\newtheorem{teo}{Theorem}

\newtheorem{problem}{Problem}
\newtheorem{altproblem}{Problem}

\newtheorem{lem}{Lemma}
\newtheorem{prop}{Proposition}
\newtheorem{rem}{Remark}

\newtheorem{corollary}{Corollary}

\include{definitions}
\newcommand{\red}{\color{black}}
\newcommand{\black}{\color{black}}
\newcommand{\blue}{\color{black}}
\newcommand{\rv}{\color{black}}

\begin{document}
	
	\title[The Reduced Basis Virtual Element Method for Eigenvalue Problems]{Avoiding stabilization terms\\in virtual elements for eigenvalue problems:\\The Reduced Basis Virtual Element Method}
	
	\author{Silvia Bertoluzza}
	\address{Istituto di Matematica Applicata e Tecnologie Informatiche \textquoteleft E. Magenes\textquoteright, Consiglio Nazionale delle Ricerche, via Ferrata 5, 27100, Pavia, Italy }
	\email{silvia.bertoluzza@imati.cnr.it}
	
	\author{Fabio Credali}
	\address{Computer, electrical and mathematical sciences and engineering division, King Abdullah University of Science and Technology, Thuwal 23955, Saudi Arabia}
	\email{fabio.credali@kaust.edu.sa}
	
	\author{Francesca Gardini}
	\address{Dipartimento di Matematica \textquoteleft F. Casorati\textquoteright, Universit\`a degli Studi di Pavia, via Ferrata 5, 27100, Pavia, Italy}
	\email{francesca.gardini@unipv.it}
	
	\maketitle
	
	\begin{abstract}
		We present the novel Reduced Basis Virtual Element Method (\textsf{rb}VEM) for solving the Laplace eigenvalue problem. This approach is based on the \rv lowest-order \black virtual element method and exploits the reduced basis technique to obtain an explicit representation of the virtual (non-polynomial) contribution to the discrete space. \textsf{rb}VEM yields a fully conforming discretization of the considered problem, so that stabilization terms are avoided. We prove that \textsf{rb}VEM provides the correct spectral approximation with optimal error estimates. Theoretical results are supplemented by an exhaustive numerical investigation.
		
		\
		
		\noindent\textbf{Keywords:} virtual element method,  eigenvalue problems, polygonal meshes, error analysis\\
		\textbf{MSC:} 65N15, 65N25, 65N30
	\end{abstract}

	\section*{Introduction}
	
	Virtual element methods (VEM) are a popular family of numerical schemes originated from mimetic finite difference~\cite{da2014mimetic} and allowing for a variational formulation in the spirit of the finite element method, see~\cite{beirao2013basic,vem-book,acta-vem} and references therein. Using VEM, the computational domain can be partitioned by general polytopal meshes, with very minimal shape regularity assumptions. At the element level, the discrete space is defined by considering polynomials up to degree $k\ge1$ plus additional \textit{virtual} functions, which are not explicitly known since they are solution of a local PDE. \rv Thus, polynomial projections are \black employed for computing the quantities of interest and ensure the optimal accuracy. On the other hand, the non-polynomial part is used to enforce \rv the conformity of the method. \black The stability of \rv the discrete bilinear form is ensured \black by introducing a suitable stabilization term~\cite{stab-survey}.
	
	In general, the stabilization term is ``artificial'' since it does not take into account the physical properties of the problem under consideration. Indeed, it is only required to scale as the virtual (non-polynomial) contribution to the stiffness matrix. It is well-known that the need for the stabilization term may represent a drawback for the method. For instance, defining such term may be non-trivial when dealing with complex nonlinear equations~\cite{wriggers2017efficient}, and it may produce inaccurate results when solving anisotropic problems~\cite{berrone2022comparison}.
	
	In order to overcome the difficulties we mentioned above, several recipes have been proposed in literature. The so called \textit{stabilization free} virtual elements~\cite{stabfree1,stabfree2} enlarge the underlying polynomial space to obtain a discrete formulation which is automatically stable. \rv A similar idea is used in~\cite{extrinsic}, where the polynomial projection is constructed onto a Raviart--Thomas space based on sub-triangulations. \black On the other hand, the \textit{self-stabilized} approach~\cite{self1,self2,self3} introduces some additional internal degrees of freedom which can be easily condensed. Another option is given by the efficient computation of the virtual functions, which can be pursued by support methods, such as the reduced basis approach~\cite{credali}, neural networks~\cite{neural-vem}, or rational functions~\cite{zerb1,zerb2}. \rv We also mention the work in~\cite{bem-fem}, where boundary elements are used to determine the basis functions of a polygonal finite element method. \black
	
	When dealing with eigenvalue problems, the issue of choosing a suitable stabilization term is even more critical since it concerns not only  the stiffness matrix, but also the mass term at the right hand side. Indeed, as observed in literature~\cite{BGG-calcolo,BGG-book}, an inappropriate choice of such stabilization terms may introduce additional artificial eigenmodes. In this context, \textit{stabilization free} and \textit{self-stabilized} methods are even more appealing. One option consists in keeping the stabilization term only at the left hand side~\cite{alzaben,BGG25}. Otherwise, fully stabilization-free approaches may be employed in the spirit of~\cite{self1}, as discussed in~\cite{Mengetal,mora-marcon}.
	
	In this work, we propose an alternative approach to those mentioned above. We resort to the reduced basis method to efficiently \rv approximate \black the local PDE defining the virtual basis functions as described in~\cite{credali} for the lowest order VEM. In this way, we obtain an explicit representation of the non-polynomial contribution to the discrete bilinear forms, so that stability is automatically ensured. This new approach results in a fully conforming polygonal method in the spirit of~\cite{manzini2014new}, where the VEM machinery is used for dealing with the polynomial part, while the reduced basis method takes care of the non-polynomial contribution. For this reason, we decided to name our approach as \textit{Reduced Basis Virtual Element Method} (\textsf{rb}VEM). We thus derive the variational formulation of the Laplace eigenvalue problem within the \textsf{rb}VEM framework and we carry out the error analysis by proving the correct spectral approximation and optimal convergence estimates. To this aim, a priori error estimates for the source problem are also proved in $H^1-$ and $L^2-$ norms. The robustness of our method and its approximation properties are confirmed by a wide range of numerical examples.
	
	The paper is organized as follows. We first recall the continuous formulation of the Laplace eigenproblem in Section~\ref{sec:pb_setting} and its standard lowest order virtual element discretization in Section~\ref{sec:vem1}. Section~\ref{sec:rbvem} discusses the construction of the new \textsf{rb}VEM space and introduces the new discrete formulation. The convergence analysis for both the source and the eigenvalue problems is carried out in Sections~\ref{sec:spec_theory}, \ref{sec:error} and~\ref{sec:error_eig}. Numerical tests confirming our theoretical findings are collected in Section~\ref{sec:tests}. Finally, we draw some conclusions in Section~\ref{sec:conclusions}.
	
	\section*{Notation}
	
	Throughout the paper we will use standard functional analysis notation. Given an open bounded domain $D\subset\R^2$, we denote by $L^2(D)$ the space of square integrable functions endowed with the scalar product $(\cdot,\cdot)_D$ inducing the norm $\|\cdot\|_{0,D}$. The symbol $\Hr{s}{D}$ stands for the family of (Hilbert) Sobolev spaces with differentiability index $s\in\R$, norm $\|\cdot\|_{s,D}$ and semi-norm $|\cdot|_{s,D}$. Moreover $H^1_0(D)$ is the subspace of $H^1(D)$ of functions with zero trace along $\partial D$. \blue Similarly, $\Hr{s}{\partial D}$ denotes the boundary Sobolev space endowed with norm $\|\cdot\|_{s,\partial D}$ and semi-norm $|\cdot|_{s,\partial D}$. \black The space of continuous functions in $D$ is denoted by $C^0(D)$ \blue with the maximum norm $\|\cdot\|_{\infty,D}$\black, whereas $\mathbb{P}_k(D)$ stands for the space of polynomials up to degree $k$ in $D$. Finally, $\mathcal{L}(V,W)$ is the space of linear and continuous operators from $V$ to $W$, with the simplified notation $\mathcal{L}(V)$ for operators acting from $V$ to itself.
	
	\section{Problem setting}\label{sec:pb_setting}
	
	In this paper, we consider the two-dimensional Laplace eigenvalue problem with homogeneous Dirichlet boundary conditions. Let $\Omega\subset\R^2$ be a polygonal domain, we seek for eigenvalues $\eig\in\R$ and eigenfunctions $\eigf\neq0$ satisfying
	\begin{equation}
		\left\{
		\begin{aligned}
			&-\Delta\eigf  = \eig\eigf&&\text{in }\Omega,\\
			&\eigf=0&&\text{on }\bOmega.
		\end{aligned}
		\right.
	\end{equation}
	
	By testing the equation with $v\in\Hunozero$, integration by parts yields the following variational formulation.
	\begin{problem}\label{pb:weak}
		Find $(\eig,\eigf)\in\R\times\Hunozero$, $\eigf\neq0$, such that
		\[
		a(\eigf,v) = \eig b(\eigf,v)\qquad\forall v\in\Hunozero,
		\]
		where
		\[
		a(\eigf,v)=(\grad\eigf,\grad v)_\Omega,\qquad b(\eigf,v)=(\eigf, v)_\Omega.
		\]
	\end{problem}
	
	The bilinear form $a(\cdot,\cdot)$ is continuous and coercive. The above problem admits an infinite sequence of strictly positive eigenvalues
	\[
	0<\eig_1\le\eig_2\le\cdots\le\eig_i\le\cdots,
	\]
	which may be repeated accordingly to their multiplicity. It is well-known that such sequence is divergent. Moreover, each $\eig_i$ is associated with an eigenfunction $\eigf_i$ satisfying the following properties
	\begin{equation}\label{eq:properties}
		\begin{aligned}
			&a(\eigf_i,\eigf_j) = b(\eigf_i,\eigf_j)= 0\quad\text{if }i\neq j,\\
			&b(\eigf_i,\eigf_i) = 1,\quad a(\eigf_i,\eigf_i)=\eig_i.
		\end{aligned}
	\end{equation}
	
	\section{The lowest order virtual element method for eigenvalue problems}\label{sec:vem1}
	
	\subsection{The numerical method}
	
	We discretize Problem~\ref{pb:weak} by the lowest order virtual element method. We thus decompose the domain $\Omega$ by a mesh $\mesh$ made up of polygons. We denote by $\element\in\mesh$ the generic mesh element, having diameter $h_\element$. As usual, the mesh size $h$ corresponds to the maximum of all $h_\element$. Moreover, by abuse of notation, we write $\edge\in\boK$ meaning that $\edge$ is an edge of  the element  $\element$. We assume standard mesh regularity properties for $\element$ (see e.g.~\cite{beirao2013basic}), requiring that each $\element$ is a star-shaped polygon with respect to a ball of radius $\rho_\element\ge C h_\element$ with~$C$~being a positive constant. Furthermore, we assume that the edge length is uniformly bounded from below by the diameter, i.e. that there exists a positive constant $\tilde{C}$, independent of $h$, such that for all~$\element\in\mesh$,  $|\edge|\ge\tilde{C} h_\element$ for all $\edge\in\partial\element$.
	
	Before defining the local discrete space, we introduce the elliptic projection $\PiNabla:\HunoK\rightarrow\PunoK$, defined as the unique solution of
	\begin{equation}\label{eq:pinabla}
		\big(\grad(\PiNabla v-v),\grad p\big)_\element=0\qquad\forall p\in\PunoK,\qquad\int_{\boK} \PiNabla v\,\ds = \int_{\boK} v\,\ds.
	\end{equation}
	
	We then consider the \textit{enhanced} virtual element space of order one on $\element$~\cite{ahmad2013equivalent}, that is
	\begin{equation}\label{eq:loc_vem}
		\begin{aligned}
			\vunoloc = \big\{ v\in\HunoK:\,
			v_{|\edge}\in\Punoe\;\forall\edge\in\boK,\,\Delta v\in\PunoK,\,
			(\PiNabla v-v,p)_\element=0\; \forall p\in\PunoK\big\}.
		\end{aligned}
	\end{equation}
	The global virtual element space is then easily obtained by glueing together all local spaces by continuity, indeed we have
	\begin{equation}
		\Vh = \{v\in\Hunozero:\,v_{|\element}\in\vunoloc\quad\forall\element\in\mesh\}.
	\end{equation}
	
	By construction, the inclusion $\PunoK\subset\vunoloc$ holds, so that optimal approximation properties are guaranteed. Furthermore, $v\in\vunoloc$ is uniquely determined by its value at the vertices of $\element$, which are a unisolvent set of degrees of freedom. We observe that $\PiNabla$ is well defined on $\HunoK$ and it is computable through the degrees of freedom when applied to $v\in\vunoloc$. Notice also that the definition of $\vunoloc$ allows also the computation of $L^2-$ projection ${\PiZero:\vunoloc\rightarrow\PunoK}$
	\begin{equation}
		(\PiZero v,p)_\element = (v,p)_\element\qquad\forall p\in\PunoK
	\end{equation}
	through the degrees of freedom. In particular, the last condition in~\eqref{eq:loc_vem} implies that $\PiZero=\PiNabla$. In general, this property is not anymore true if we increase the polynomial degree of the discrete space. In the remainder of the paper, we will keep the two notation distinct.
	
	We now discuss the construction of the discrete problem. We first write the bilinear forms~$a$ and~$b$ in terms of local contributions, so that
	\begin{equation}
		\begin{aligned}
			&a(u,v) = \sum_{\element\in\mesh} \aK(u,v),&&\qquad\aK(u,v)=(\grad u,\grad v)_\element,\\
			&b(u,v) = \sum_{\element\in\mesh} \bK(u,v),&&\qquad\bK(u,v)=(u,v)_\element.
		\end{aligned}
	\end{equation}
	
	The projectors $\PiNabla$ provide an orthogonal decomposition of $\vunoloc$ into a polynomial space plus the projection kernel. More precisely, we have
	\begin{equation}
		\vunoloc = \PunoK\oplus\Vperp,\qquad\text{where }\Vperp=ker(\PiNabla)\subset\vunoloc.
	\end{equation}
	Therefore, the splitting of $u_h,v_h\in\vunoloc$ into a polynomial part plus a non-polynomial part gives
	\begin{equation}\label{eq:splitting}
		\begin{aligned}
			&\aK(u_h,v_h)=\aK(\PiNabla u_h,\PiNabla v_h)+\aK( u_h-\PiNabla u_h,v_h-\PiNabla v_h)\\
			&\bK(u_h,v_h)=\bK( \PiZero u_h, \PiZero v_h)+\bK( u_h-\PiZero u_h,v_h-\PiZero v_h)
		\end{aligned}
	\end{equation}
	where only the first term at the right hand side of both equations is computable through the degrees of freedom. Such terms ensure the \emph{consistency} of the numerical method. On the other hand, stability is ensured by replacing the non-computable contributions with suitable stabilization terms. Let $\staba,\,\stabb$ be any two symmetric positive definite bilinear forms scaling as $\aK$ and $\bK$, respectively
	\begin{equation}
		\begin{aligned}
			&c_1^\star \aK(v_h,v_h) \le \staba(v_h,v_h) \le c_1^\sharp \aK(v_h,v_h)&&\forall v_h\in\Vperp,\\
			&c_2^\star\bK(v_h,v_h) \le \stabb(v_h,v_h) \le c_2^\sharp \bK(v_h,v_h)&&\forall v_h\in\Vperp
		\end{aligned}
	\end{equation}
	for some positive constants $c_i^\star,c_i^\sharp$ for $i=1,2$. We then define the local discrete bilinear forms as
	\begin{equation}\label{eq:discrete_bils}
		\begin{aligned}
			&\ahK(u_h,v_h)=\aK(\PiNabla u_h,\PiNabla v_h)+\staba( u_h-\PiNabla u_h,v_h-\PiNabla v_h),\\
			&\bhK(u_h,v_h)=\bK( \PiZero u_h, \PiZero v_h)+\stabb( u_h-\PiZero u_h,v_h-\PiZero v_h).
		\end{aligned}
	\end{equation}
	The global forms are obtained by summing all local contributions, i.e. $\ah(u,v) = \sum_{\element\in\mesh} \ahK(u,v)$ and $\bh(u,v) = \sum_{\element\in\mesh} \bhK(u,v)$.
	The discrete problem reads as follows.
	
	\begin{problem}\label{pb:vem}
		Find $(\eig_h,\eigf_h)\in\R\times\Vh$, $\eigf_h\neq0$, such that
		\[
		\ah(\eigf_h,v_h) = \eig_h \bh(\eigf_h,v_h)\qquad\forall v_h\in\Vh.
		\]
	\end{problem}
	
	Thanks to the stability property, Problem~\ref{pb:vem} admits $\neig=\dim\Vh$ positive eigenvalues
	\[
	0<\eig_{1,h}\le\dots\le\eig_{\neig,h}.
	\]
	We denote by $\eigf_{i,h}$ with $i=1,\dots,\neig$, the corresponding eigenfunctions, satisfying the discrete counterpart of the properties reported in~\eqref{eq:properties}. 
	\rv Error estimates for \black the standard virtual element method \rv for eigenvalue problems \black has been presented in~\cite{gardini2018}.

	The stabilization term $\staba$ in the definition of the discrete bilinear form $a_h$ is a key ingredient when designing virtual element discretization. Indeed, it guarantees stability and well-posedness of the discrete source problem~\cite{beirao2013basic,ahmad2013equivalent}.
	
	When dealing with eigenvalue problems, the mass term $b(\cdot,\cdot)$ appears at the right hand side, posing the issue of its discretization in virtual element framework. A straightforward approach, as discussed above, approximates $b$ by $b_h$ following the same idea as for $a_h$, namely introducing a second stabilization term $\stabb$ (see e.g.~\cite{steklov,acoustic-2017,mora2018,gardini2018,gardini2019}).
	
	Nevertheless, another approach has been proposed in literature when analyzing the use of virtual elements for the approximation of eigenvalue problems~\cite{Mengetal,alzaben,mora-marcon,BGG25}. Such approach consists in dropping the term $\stabb$ from the definition of the discrete mass term~$b_h$. 
	
	In~\cite{BGG-calcolo}, the authors investigated the effect of the stabilization term $\stabb$; in particular, they showed that the choice of the stabilization parameter is critical since it may introduce spurious eigenvalues polluting the discrete spectrum. For the sake of completeness, we briefly recall such results in the next section.
	
	We point out that several options are available for defining the stabilization terms $\staba$ and $\stabb$, also depending on the problem under consideration, see e.g.~\cite{cangiani2015hourglass,mascotto2018ill,bertoluzza2022stabilization}. 
	{\red Previous work on the use of virtual elements for eigenvalue approximation follows}~\cite{BGG-calcolo} and make the simplest choice by considering the \textit{dofi--dofi} stabilization: letting $\vertex_1,\dots,\vertex_\Nvert$ be the vertices of the generic element $\element\in\mesh$, this is defined as
	\begin{equation}\label{eq:stab-forms}
		\staba(u_h,v_h)=\para_\element\sum_{i=1}^{\Nvert}u_h(\vertex_i)v_h(\vertex_i),
		\qquad
		\stabb(u_h,v_h)=\parb_\element h_\element^2\sum_{i=1}^{\Nvert}u_h(\vertex_i)v_h(\vertex_i),
	\end{equation}
	where $\para_\element$ and $\parb_\element$ are two positive stability parameters, which may depend on $\element$, but not on $h$. The value of $\para_\element$ and $\parb_\element$ is usually chosen in dependence of the eigenvalues of the local matrices arising from the discretization of $\aK(\PiNabla u_h,\PiNabla v_h)$ and $(1/h_\element^2)\bK(\PiZero u_h,\PiZero v_h)$, respectively. \rv A reasonable choice is given, for instance, by the mean of the eigenvalues of the consistency term~\cite{gardini2018}. \black
	
	The aim of the present work is {\red instead} to avoid any possible issue regarding the presence of stabilization terms for both discrete bilinear forms. In this direction, {\red we will introduce reduced basis approximation of the nonpolynomial contribution $\Vperp$ to the virtual element space $\vunoloc$,  in the spirit of~\cite{credali}, which will allow us to compute the non polynomial component of the splitting \eqref{eq:splitting}. This is equivalent to introduce a new discrete space somehow equivalent to the virtual element space \eqref{eq:loc_vem}, which we will leverage for the analysis of the method.}
	
	\subsection{Parameter-dependent generalized eigenvalue problems in matrix form}
	
	In this section, \rv for the sake of completeness, \black we recall the main features of the numerical investigation presented in~\cite{BGG-calcolo}, dealing with the effect of stabilization parameters on the approximation of eigenvalue problems.
	
	To fix ideas, we now assume that $\para_\element=\alpha$ and $\parb_\element=\beta$ in~\eqref{eq:stab-forms} for all polygons $\element\in\mesh$.	This assumption is reasonable if the parameters $\para_\element,\,\parb_\element$ vary in a small range.  
	
	Due to the presence of  the stabilization parameters in~\eqref{eq:stab-forms}, and thus in the discrete forms 
	$\ah$ and $\bh$, it turns out that the algebraic  generalized eigenvalue problem
	\begin{equation}\label{eq:matrix-form}
		\stiff(\alpha)\,\m{u}=\m{\lambda}\, \mass(\beta)\,\m{u},
	\end{equation}
	arising from Problem~\ref{pb:vem}, is parameter dependent. We point out that $\stiff(\alpha)=\stiff_1+\alpha\stiff_2\in\R^{\neig\times\neig}$ represents the stiffness matrix associated with $\ah$, whereas $\mass(\beta)=\mass_1+\beta\mass_2\in\R^{\neig\times\neig}$ is the mass matrix associated with $\bh$.
	
	Recently, it has been observed that the stabilizing parameters may introduce artificial eigenmodes which pollute the spectrum~\cite{BGG-calcolo,BGG-book}. In this context, the theoretical results state that, for a given choice of the stabilization parameters, the discrete solution converges to the continuous one with optimal order in $h$ or/and $p$ (see~\cite{gardini2018,certik}). It is implicitly understood that the convergence depends on the choice of the parameters. 
	
	In~\cite{BGG-calcolo}, under suitable assumptions on $\stiff_1,\stiff_2,\mass_1,\mass_2$, the authors discuss how the computed spectrum varies in terms of $\alpha$ and $\beta$. Moreover, several numerical examples show that the parameters have to be carefully tuned and that wrong choices produce useless results. 
	
	More precisely, if only one parameter is varying, the eigenvalues split into two families: a set of eigenvalues depends on such parameter, while all the others are independent. In particular, the dependence on $\alpha$ is linear and thus the corresponding eigenvalues lie on a straight line, while those depending on $\beta$ lie on hyperbolas. This behavior is depicted in Figure~\ref{fig:param_calcolo}, where a simple example is considered. We introduce the matrices $\Cmatr_1=\mathsf{diag}([3,4,5,6,0,0])$ and $\Cmatr_2=\mathsf{diag}([0,0,0,0,1,2])$. On the left panel, we set $\stiff(\alpha)=\Cmatr_1+\alpha\Cmatr_2$ and $\mass=\eye(6)$, so that the eigevalues are given by $\alpha,2\alpha,3,4,5,6$. On the right panel, we set $\stiff=\eye(6)$ and $\mass(\beta)=\Cmatr_1+\beta\Cmatr_2$, so that the eigenvalues are $1/\beta,1/(2\beta),1/3,1/4,1/5,1/6$. It is clear that, in the first case, the eigenvalues depending on $\alpha$ are aligned, whereas in the second case, eigenvalues depending on $\beta$ lie on hyperbolas.
	
	\begin{figure}
		\includegraphics[width=0.45\linewidth]{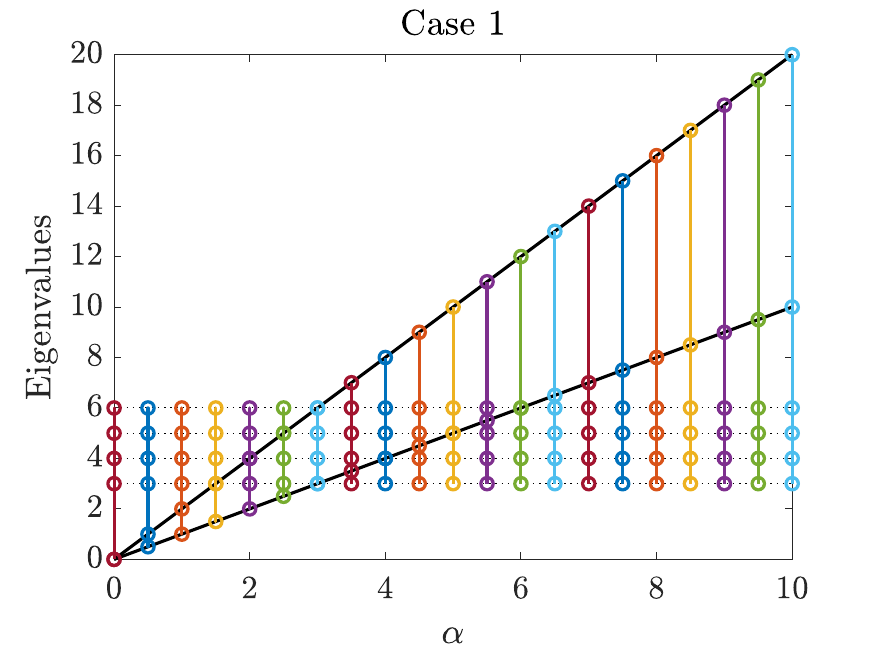}\quad
		\includegraphics[width=0.45\linewidth]{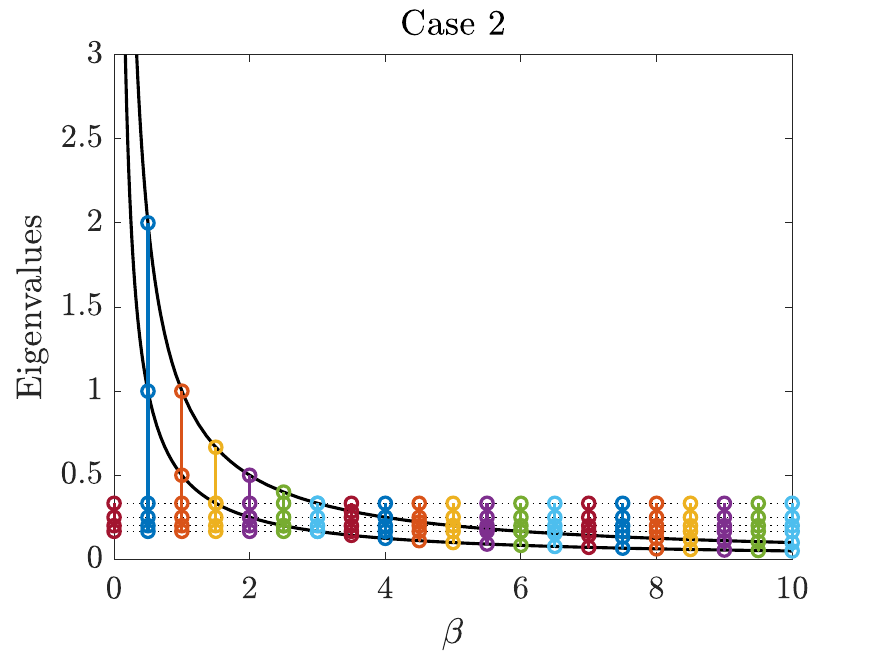}
		\caption{Parameters dependence of the eigenvalues.}
		\label{fig:param_calcolo}
	\end{figure}
	
	Even if the matrices arising from the VEM discretization do not completely satisfy the requirements in~\cite[Assump. 1]{BGG-calcolo}, numerical tests showed a similar parameters dependence. Choosing a sufficiently large stabilization parameter $\alpha$ and a small (possibly zero) $\beta$ could be the safest choice. Nevertheless, when $\beta=0$, the discrete form $b_h(\cdot,\cdot)$ equals the consistency term $b(\PiZero\cdot,\PiZero\cdot)$, and thus it generally has a nontrivial kernel. In this case, the partially stabilized VEM provides numerical optimal results even if the mass term does not enjoy the stability property. 
	We refer to~\cite{BGG25} for the analysis of first and second order VEM approximations of eigenvalue problems when the stabilization parameter $\beta$ is zero. 
	
	Taking into account the above discussion, in order to avoid any dependence of the discrete spectrum on parameters, in this paper we resort to the \textsf{rb}VEM approach introduced in~\cite{credali} for the source problem.
	
	\section{The \textsf{rb}VEM discrete space}\label{sec:rbvem}
	
	In order to overcome the issues relative to the stabilization forms, we follow the approach introduced in~\cite{credali}, where stability of the discrete source problem is enforced by actually approximating the virtual contribution. We extended such approach to the mass term to deal with eigenvalue problems: this is equivalent to discretize Problem~\ref{pb:weak} in a different (conforming) space where bilinear forms are fully computable.
	
	The new discrete space is built upon the virtual element one by means of the Reduced Basis Method, which has been designed for the efficient solution of parametric equations for a large set of parameters. Indeed, we can interpret the family of local problems defining the virtual basis functions on different elements as a geometric parametric problem defined on a \rv suitable \black reference \rv polygon. \black 
	
	In this section we introduce the parametric problem and postpone to the Appendix the issue of its efficient solution. We then define the new polygonal discrete space.
	
	Let us consider a generic star-shaped polygon $\element$ with $N$ vertices $\vv_1,\dots,\vv_N$ ordered counterclockwise. All functions $v\in\vunoloc$ can be written in terms of the nodal basis functions $\base_1,\dots,\base_N$, which are solution of the following local PDE.
	
	\begin{problem}\label{pb:basis}
		For $i=1,\dots,N$, find $\base_i$ such that
		\begin{equation}
			\begin{aligned}
				-\Delta \base_i &= 0&&\text{in }\element\\
				\base_i &= g_i&&\text{on }\boK,
			\end{aligned}
		\end{equation}
		with $g_i$ such that $g_{i|\edge}\in\Punoe$ for all $\edge\in\boK$, and $g_i(\vv_j)=\delta_{ij}$ for $j=1,\dots,N$.
	\end{problem}
	
	Problem~\ref{pb:basis} can be formulated as a parametric problem on the family of star-shaped polygons with $N$ vertices. To this aim, for fixed $N$, we rewrite Problem~\ref{pb:basis} on a \rv suitable reference polygon \black (see, for instance,~\cite[Chap. 6.2]{hesthaven2016certified}). Given a polygon $\element$, let $\Ker(\element)$ denote the set of point for which $\element$ is star-shaped and let $\xK$ and $\radiusK$ being the centroid and the radius of the circumscribed circle, respectively. We then introduce the parameter space	
	\begin{equation}\label{defpolyparameters}
		\parameters = \{\element: \ \element\text{ polygon with $N$ vertices with }\Ker(\element) \not=\emptyset,\  \radiusK = 1, \ \xK  =  \hx = (0,0)^{\red \top}\}.
	\end{equation} 
	
	In other words, $\parameters$ contains all star-shaped polygons with $N$ vertices, centered in the origin, and having circumscribed circle of unit radius.
	
	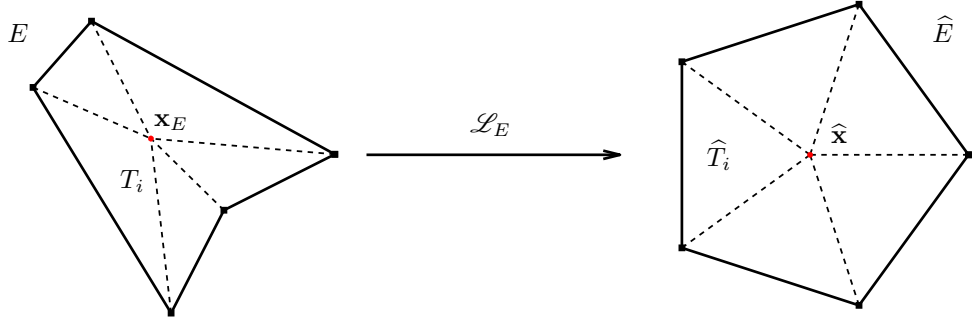
\begin{figure}
		\begin{overpic}[trim = 100 200 100 200,width=15cm]{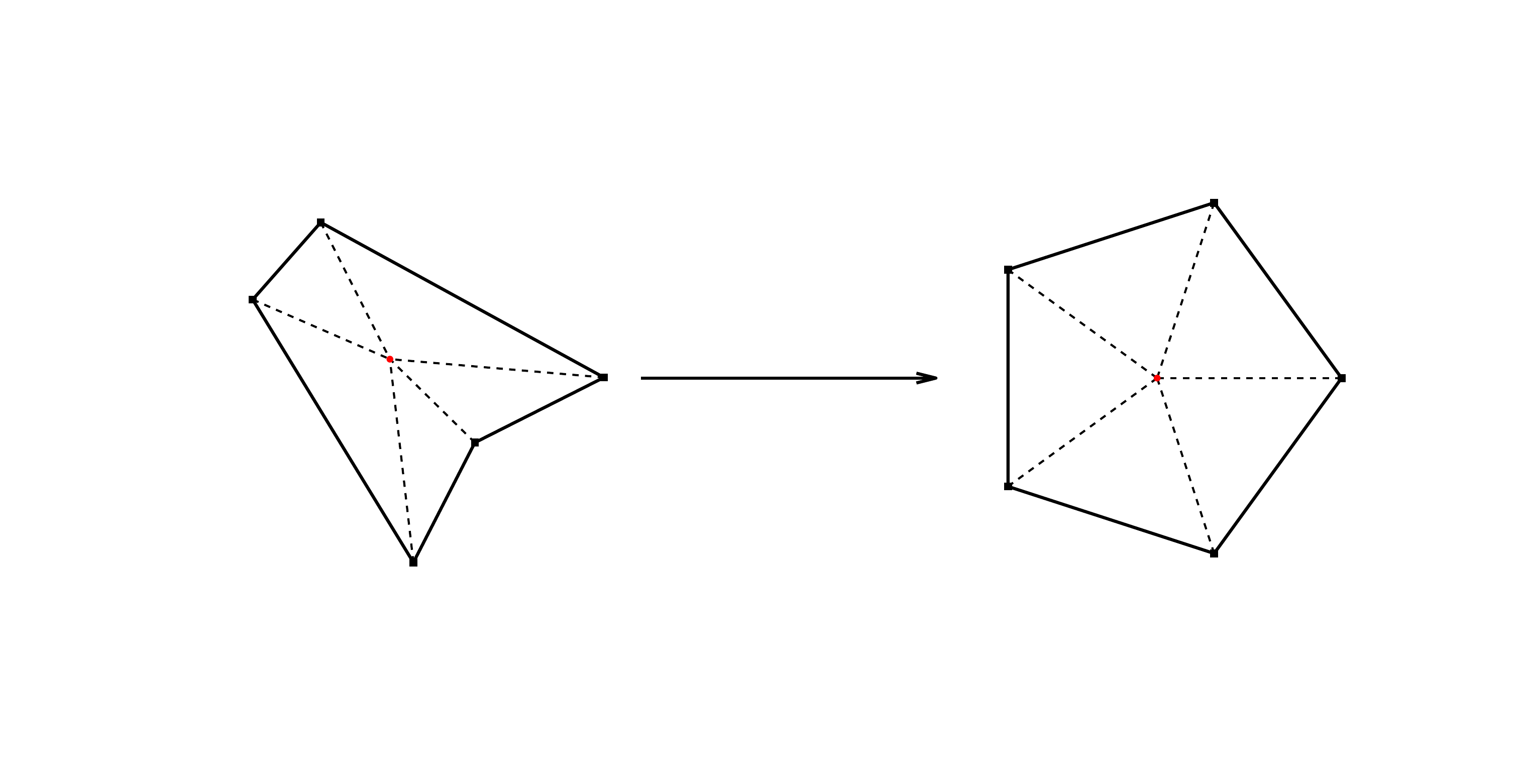}
			\put (71,14) {$\That_i$}
			\put (19,12) {$\Tri_i$}
			\put (22,17.5) {$\xK$}
			\put (82,16) {$\hx$}
			\put (50,17) {$\map$}
			\put (9,25) {$\element$}
			\put (91,25) {$\refelement$}
		\end{overpic}
		\caption{An example of \rv piecewise affine mapping \black $\map$  between a {star-shaped} pentagon and its reference configuration.}
		\label{fig:affine_mapping}
	\end{figure}
	
	As reference element, we introduce the regular polygon $\refelement$ with vertices $\hv_1,\dots,\hv_N$ (ordered counterclockwise), centered in the origin $\hx=(0,0)^{\red \top}$. We then construct a piecewise affine map $\map$ from the generic $\element$ to $\refelement$, as depicted in Figure~\ref{fig:affine_mapping}. We decompose both $\element$ and $\refelement$ into $N$ triangles by connecting their vertices with $\xK$ and $\hx$, respectively, i.e.
	\begin{equation}\label{eq:partition}
		\element = \bigcup_{i=1}^N \Tri_i \quad \text{and} \quad \refelement = \bigcup_{i=1}^N \That_i.
	\end{equation}
	\rv Thanks to our mesh regularity assumptions, the shape-regularity of $\element$ implies the shape-regularity of $\Tri_1,\dots,\Tri_N$. \black We define $\map:\element\rightarrow\refelement$ as
	\begin{equation}
		\map(\x) = \Bmatr_{\element,i}\x \quad \text{on } T_i
	\end{equation}
	such that
	\begin{equation}
		\map(\vv_i)=\hv_i,\quad i=1,\dots,N\quad\text{and}\quad\map(\rv\xK\black)=\hx.
	\end{equation}
	We point out that $\Bmatr_{\element,i}$ is a $2\times2$ matrix defined on $\Tri_i$.
	
	\begin{rem}
		\rv As already commented in~\cite{credali}, the method can also be extended to the three dimensional case. However, this requires the construction of a suitable set of reference polyhedra and piecewise linear maps. Due to the complex shape of general polyhedra, this operation is far from being trivial and one may consider to adopt domain decomposition strategies. Up to now, the three dimensional setting has not been studied and its efficiency is not guaranteed yet.\black
	\end{rem}
	
	In order to formulate the parametric problem, we first write Problem~\ref{pb:basis} in variational form and we apply a change of variables, so that we have
	\begin{equation}
		(\grad w,\grad v)_\element = \sum_{i=1}^{N} (\grad w,\grad v)_{\Tri_i}= \sum_{i=1}^{N} 
		\left(|\det(\Bmatr^{-1}_{\element,i})|\, \Bmatr_{\element,i}^\top \Bmatr_{\element,i}\,\grad\widehat{w},\grad\widehat{v}\right)_{\That_i},
	\end{equation}
	where $w$ and $v$ are two generic functions on $\element$.
	Finally, we introduce the parameter-dependent bilinear form $\aRB$ defined as
	\begin{equation}\label{eq:def_aRB}
		\aRB(\widehat w,\widehat v;\element) = \sum_{i=1}^{N} 
		\left(|\det(\Bmatr^{-1}_{\element,i})|\, \Bmatr_{\element,i}^\top \Bmatr_{\element,i}\,\grad\widehat{w},\grad\widehat{v}\right)_{\That_i},
	\end{equation}
	hence we obtain the following version of Problem~\ref{pb:basis}, defined on the reference polygon.
	
	\begin{problem}\label{pro:parametric}
		Find $\hbase_i^{\rv \element\black} \in H^1(\refelement)$ such that
		\begin{equation*}\label{eq:postprocessinghat}	
			\hbase_i^{\rv \element\black} = \widehat g_i \quad\text{ on }\partial \widehat {\rv \element\black}, \qquad \text{ and }\qquad	\aRB(\hbase_i^{\rv \element\black},\widehat v;\element)		 = 0, \quad  \text{ for all } \widehat v \in H^1_0(\refelement), 
		\end{equation*}
		where  $\widehat{g}_i$ is the piecewise linear function on $\partial \refelement$ such that $\widehat g_i(\widehat \vv_k)= \delta_{i,k}$ for $i,k=1,\dots,N$.
	\end{problem}
	
	\rv We introduce the piecewise linear finite element space $\Vdelta$ defined on a fine (fixed) triangulation $\meshdelta$ of $\refelement$.
	{ To handle the non-homogeneous boundary conditions, as customary in the finite element method, we reformulate Problem~\ref{pro:parametric} in the homogeneous setting by introducing functions ``carrying'' the boundary values, which, we recall, are independent of~$E$.}
	To this aim, \black we let $\hTheta_i\in\Vdelta$ be the discrete harmonic lifting of $\widehat g_i$ defined as solution of
	\begin{equation}
		\left(\grad\hTheta_i,\grad v_\delta\right)_{\refelement} = 0\qquad\forall v_\delta\in\Vdelta\cap\Huzh,\qquad\text{with}\quad \hTheta_i=\widehat{g}_i\quad\text{on }\partial\refelement,
	\end{equation}
	so that we can write $\hbase_i^{\rv\element\black}$ as the sum of two contributions
	\begin{equation}
		\hbase_i^{\rv\element\black} = \hTheta_i+\resKi,
	\end{equation}
	where $\resKi$ is solution of the following parametric problem \rv with homogeneous boundary conditions.\black
	
	\begin{problem}\label{pro:refK}
		For $i=1,\dots,N$, find $\resKi\in\Huzh$ such that
		$$
		\aRB(\resKi,\vhat;\element) = -\aRB(\hTheta_i,\vhat;\element)\qquad\forall\vhat\in\Huzh.
		$$
	\end{problem}
	
	\rv In order to solve the parametric Problem~\ref{pro:refK}, we resort to the Reduced Basis Method, see~\cite{hesthaven2016certified} for a survey on its foundations. Such method consists in an offline procedure, { performed once and for all}, where the computational load is concentrated, followed by an efficient online phase, where {Problem \ref{pro:refK} is solved cheaply by} making use of pre-computed objects. \black
	The reduced basis approach we designed for efficiently solve Problem~\ref{pro:refK} has been presented in~\cite{credali} and it is briefly recalled in the Appendix. In this context, we seek for an approximation \rv$\resKirb$ \black of $\resKi$ in a {\red suitable} reduced {\red sub}space {\red of $\Vdelta$} of dimension $M\ll\dim(\Vdelta)$, so that we find
	\begin{equation}\label{eq:rb_vem_basis_d}
		\baserbhat{i} = \hTheta_i + \rv\resKirb\black,\qquad i=1,\dots,N,
	\end{equation}
	and, by pull back, we obtain 
	\begin{equation}\label{eq:rb_vem_basis_2}
		\baserb{i} = \baserbhat{i}\circ\rv\map\black,\qquad i=1,\dots,N,
	\end{equation}
	which are the reduced basis approximation of the VEM basis functions.
	
	Given $\element\in\mesh$, our approach is based on the idea that we can compute exactly the bilinear forms $\aK$ and $\bK$ if we consider the functions $\base_{M,i}^\mathrm{rb}[\element]$ constructed in~\eqref{eq:rb_vem_basis_2}. Then we define a new suitable discrete space $\Vhrb$, where all quantities are computable. We first introduce the boundary space
	
	\begin{equation}
		\mathbb{B}^\perp(\boK) = \left\{v\in H^{1/2}(\boK):\,\int_{\boK}v\grad q\cdot\mathbf{n}^\element\,\ds=0\ \forall q\in\PunoK,\,\int_{\boK} v\,\ds=0\right\}
	\end{equation} 
	\rv so that the elliptic projection $\PiNabla$ defined in~\eqref {eq:pinabla} satisfies \black $\PiNabla v=0$ if and only if $v{\red |_{\boK}}\in\mathbb{B}^\perp(\boK)$. For fixed $M$, we then define the RB space 
	\begin{equation}\label{eq:wrb}
		\Whrb(\element)=\{v\in span\{\baserb{1},\dots,\baserb{N}\}:\,v_{\rv |\boK\black}\in\mathbb{B}^\perp(\boK)\}
	\end{equation}
	and, finally, the new discrete \textsf{rb}VEM space
	\begin{equation}\label{eq:rbspace}
		\Vhrb(\element) = \PunoK\oplus\Whrb(\element).
	\end{equation}
	We observe that, by construction, $\PunoK\subset\Vhrb (\element)$. Moreover, the polynomial contribution to $\Vhrb(\element)$ is the result of the action of the operator $\PiNabla$ on the virtual element space $\vunoloc$. \blue In other words, our new space $\Vhrb(\element)$ is spanned by the basis functions $\baseVrb{1},\dots,\baseVrb{N}$ defined as
	$$
	\baseVrb{i} = \PiNabla e_i + \sum_{j=1}^N \big(e_i(\vv_j)-\PiNabla e_i(\vv_j)\big)\,\baserb{j},\qquad i=1,\dots,N,
	$$
	where $e_i$ are the virtual basis functions of $\vunoloc$\black. \rv Notice that the approximation properties of the \textsf{rb}VEM space are guaranteed by the inclusion of polynomials and do not depend on how well $\baserb{j}$ approximate the true VEM basis functions. Indeed, $\baserb{j}$ are only employed to recover the non-computable contribution of the original VEM space, {needed to ensure the conformity of the discretization space}.\black
	
	Let $\uhrb,\vRB\in\Vhrb(\element)$, then the following representation holds
	\begin{equation}
		\uhrb = u^P + u^W,\qquad\vRB = v^P + v^W,\qquad\text{with }u^P,v^P\in\PunoK,\, u^W,v^W\in\Whrb(\element)
	\end{equation}
	so that
	\begin{equation}\label{eq:rbstiff}
		a(\uhrb,\vRB) = a(u^P,v^P) + a(u^W,v^W) + a(u^P,v^W) + a(u^W,v^P)
	\end{equation}
	and 
	\begin{equation}\label{eq:rbmass}
		b(\uhrb,\vRB) = b(u^P,v^P) + b(u^W,v^W) + b(u^P,v^W) + b(u^W,v^P).
	\end{equation}
	We observe that due to the orthogonal decomposition of the local \textsf{rb}VEM space $\Vhrb(\element)$ with respect to $a(\cdot,\cdot)$, the cross terms $a(u^P,v^W)$ and $a(u^W,v^P)$ vanish. On the other hand, we emphasize that this is not true for $b(u^P,v^W)$ and $b(u^W,v^P)$, which must be considered when assembling the mass matrix. The computability of such terms is not an issue since \textsf{rb}VEM provides all the required information. Algorithm~\ref{alg:method} in the Appendix describes how to efficiently construct the matrices arising from Problem~\ref{pb:rb_eig} by exploiting the properties of VEM and of the Reduced Basis method. We are going to show that the new space $\Vhrb(\element)$ satisfies optimal approximation properties and that the solution of Problem~\ref{pb:rb_eig} is not polluted by spurious eigenvalues.
	
	By glueing all local spaces by continuity, we find the global \textsf{rb}VEM space
	\begin{equation}
		\Vhrb = \{v\in\Hunozero:\,v_{|\element}\in\Vhrb(\element)\,\text{for all }\element\in\mesh\}.
	\end{equation}
	Discretizing Problem~\ref{pb:weak} in $\Vhrb$ by a standard (fully conforming) Galerkin method we obtain the following problem.
	\begin{problem}\label{pb:rb_eig}
		Find $(\eigRB,\eigfRB)\in\R\times\Vhrb$, $\eigfRB\neq0$, such that
		\[
		a(\eigfRB,\vRB) = \eigRB b(\eigfRB,\vRB)\qquad\forall \vRB\in\Vhrb.
		\]
	\end{problem}
	Since basis functions are explicitly known, all quantities are computable and stabilization parameters disappear from the discrete problem.
	{\red We highlight that, by construction, the new discrete space has the same set of degrees of freedom as the original virtual element space. We can then interpret Problem \ref{pb:rb_eig} as a different non-conforming formulation of our eigenvalue problem in the virtual element space $\Vh$, of the form
		\begin{altproblem}\label{pb:equiv_vem_eig}	Find $(\eigvem,\eigfvem)\in\R\times\Vh$, $\eigfvem\neq0$, such that
			\[
			a_h(\eigfvem,v_h) = \eigvem b_h(\eigfvem,v_h)\qquad\forall v_h\in\Vh,
			\]
			where $a_h: \Vh \times \Vh \to \mathbb{R}$ and $b_h:\Vh \times \Vh \to  \mathbb{R}$ are defined as
			\begin{gather}
				a_h(u_h,v_h) = \sum_\element a_h^\element(u_h,v_h), \qquad a_h^\element(\base_i,\base_j) = a(\blue\baseVrb{i},\baseVrb{j}\black),\\
				b_h(u_h,v_h) = \sum_\element b_h^\element(u_h,v_h), \qquad b_h^\element(\base_i,\base_j) = b(\blue\baseVrb{i},\baseVrb{j}\black).
			\end{gather}
		\end{altproblem}	
	}
	
	\begin{rem}\label{rem:rbstab}
		The ``virtual element reader'' would have probably discretized the eigenvalue problem in a different way by considering the (local) bilinear forms
		\begin{equation}\label{eq:rb-stab}
			\begin{aligned}
				&\ahK(u_h,v_h) = a(\PiNabla u_h,\PiNabla v_h) + S_a^{\element,\mathrm{rb}}(u_h^W,v_h^W)
				\qquad
				\text{with }\quad S_a^{\element,\mathrm{rb}}(u_h^W,v_h^W) = \aK(u_h^W,v_h^W),
				\\
				&\bhK(u_h,v_h) = b(\PiZero u_h,\PiZero v_h) + \chi\,S_b^{\element,\mathrm{rb}}(u_h^W,v_h^W)
				\qquad
				\text{with }\quad S_b^{\element,\mathrm{rb}}(u_h^W,v_h^W) = \bK(u_h^W,v_h^W)
			\end{aligned}
		\end{equation} 
		with $u_h,v_h\in\vunoloc$, $u_h^W,v_h^W\in\Whrb(\element)$ and $\chi=0,1$.
		
		More precisely, $\ahK(\cdot,\cdot)$ is defined by choosing the RB-stabilization introduced in~\cite{credali} and the stability parameter is dropped out since we are replacing the non-computable term by its reduced basis approximation. Notice that $\ahK(\cdot,\cdot)$ is exactly the same bilinear form as in~\eqref{eq:rbstiff}, but defined in the original virtual element space $\vunoloc$. The same argument applies to the mass term: indeed, if $\chi=1$, we stabilize $\bhK(\cdot,\cdot)$ by a reduced basis approximation of the non-computable part of $\vunoloc$. Unlike the bilinear form in~\eqref{eq:rbmass}, this definition does not contain cross terms. The choice $\chi=0$ yields to a partially stabilized formulation that is also admissible.
		
		The \emph{rb--stab--VEM} formulations in~\eqref{eq:rb-stab} fall in the framework of standard VEM, where stabilization terms are only required to scale as the actual non-computable contributions. Optimal convergence properties are thus ensured by the virtual elements theory.
	\end{rem}
	
	\section{Spectral theory for compact operators}\label{sec:spec_theory}

	We first introduce the continuous source problem associated with Problem~\ref{pb:weak}, which reads as follows.
	\begin{problem}\label{pb:source_cont}
		Given $f\in\Ldue$, find $\u\in\Hunozero$ such that
		\[
		a(\u,v) = (f,v)_\Omega\qquad\forall v\in\Hunozero.
		\]
	\end{problem}
	It is well-known that the above problem is well-posed and admits a unique solution~\cite{lions2012non}. Moreover, there exists $1/2<s\le1$ such that $\u\in\Hr{1+s}{\Omega}$ and the following stability estimate holds
	\begin{equation}\label{eq:stability}
		\|\u\|_{1+s,\Omega} \le C\,\|f\|_{0,\Omega}.
	\end{equation}
	In order to carry out the spectral convergence analysis of our method, we introduce the solution operator $\so:\Hr{1}{\Omega}\rightarrow\Hr{1+s}{\Omega}\subset\Hr{1}{\Omega}$ defined as $\so f = \u$, where $\u$ is the unique solution to Problem~\ref{pb:source_cont}, namely 
	\[
	a(\so f,v) = (f,v)_\Omega\qquad\forall v\in\Hunozero.
	\]
	It turns out that such operator is linear, continuous, and self-adjoint. We point out that $\so$ is also compact, thanks to the compact embedding of $\Hr{1+s}{\Omega}$ in $\Hr{1}{\Omega}$. Moreover, there is a one-to-one correspondence between eigenpairs of Problem~\ref{pb:weak} and those of $\so$. Indeed, if $(\eig,\eigf)$ is an eigenpair for the continuous Problem~\ref{pb:weak}, then $(1/\eig,\eigf)$ is an eigenpair for $\so$, and viceversa.
	
	Similarly, we now introduce the discrete source problem associated with Problem~\ref{pb:rb_eig}, with both bilinear form $a$ and right hand side discretized by the \textsf{rb}{VEM} approach.
	
	\begin{problem}\label{pb:source_rb}
		Given $f\in\Ldue$, find $\uhrb\in\Vhrb$ such that
		\[
		a(\uhrb,\vRB) = (f,\vRB)_\Omega\qquad\forall \vRB\in\Vhrb.
		\]
	\end{problem}
	
	{\red	By looking at \textsf{rb}VEM  as a standard Galerkin approximation in $\Vhrb$,  we can avoid dealing with the consistency error that arises from the non conformity that naturally arises in the virtual element framework.}
	
	We observe that continuity and coercivity of the bilinear form $a$, together with the linearity of the right hand side, ensure the well-posedness of Problem~\ref{pb:source_rb}, as stated by the Lax--Milgram theorem. The discrete solution operator ${\so_h:\Hr{1}{\Omega}\rightarrow\Vhrb\subset\Hr{1}{\Omega}}$ is thus defined as $\so_h f = \uhrb$, with $\uhrb$ being the unique solution to Problem~\ref{pb:source_rb}. In other words, we have
	\[
	a(\so_h f,\vRB) = (f,\vRB)_\Omega\qquad\forall \vRB\in\Vhrb.
	\]
	It holds that the eigenpairs of Problem~\ref{pb:rb_eig} and $\so_h$ are related as in the continuous case. In addition, the operator $\so_h$ is compact since its range is a finite dimensional space.
	
	As a consequence, we resort to the theory for spectral approximation of compact operators~\cite{boffi-acta,babuska-osborn} to carry out the error analysis for our numerical method. Nevertheless, in the context of eigenvalue problems, the mere convergence of the numerical method is not enough to guarantee a correct approximation of the spectrum. Indeed, besides convergence of the discrete eigenpairs, the \emph{correct spectral approximation} requires that no spurious eigenvalues pollute the spectrum, i.e.
	\begin{itemize}[-]
		\item each discrete eigenvalue converges to a continuous one (preserving the ordering);
		\item each continuous eigenvalue with multiplicity $m$ is approximated by exactly $m$ discrete eigenvalues, counted with their multiplicity. 
	\end{itemize}
	
	It is well known~\rv\cite{babuska-osborn,boffi-acta} \black that a sufficient condition for obtaining a correct approximation is the uniform convergence of the sequence $\so_h$ to $\so$ as $h\rightarrow0$, namely
	\begin{equation}\label{eq:unif_conv}
		\|\so-\so_h\|_{\mathcal{L}(\Hr{1}{\Omega})} = \sup_{f\in\Hr{1}{\Omega}} \frac{\|(\so-\so_h)f\|_{1,\Omega}}{\|f\|_{1,\Omega}} \longrightarrow 0 \quad\text{for}\quad h\rightarrow0.
	\end{equation}
	We emphasize that, \rv in our context, \black the uniform convergence~\eqref{eq:unif_conv} directly follows from an \textit{a priori} error estimate for the source problem of kind
	\begin{equation}\label{eq:h1_a_priori}
		\| \eigf - \uhrb \|_{1,\Omega}
		\le \mathfrak{E}(h)\|f\|_{0,\Omega} \le \mathfrak{E}(h)\|f\|_{1,\Omega}\qquad\text{with}\quad\mathfrak{E}(h)\rightarrow0\quad\text{for}\quad h\rightarrow0,
	\end{equation}
	indeed 
	\begin{equation}
		\|\so-\so_h\|_{\mathcal{L}(\Hr{1}{\Omega})} = \sup_{f\in\Hr{1}{\Omega}} \frac{\|(\so-\so_h)f\|_{1,\Omega}}{\|f\|_{1,\Omega}} = \sup_{f\in\Hr{1}{\Omega}}\frac{\| \eigf - \uhrb \|_{1,\Omega}}{\|f\|_{1,\Omega}} \le \mathfrak{E}(h).
	\end{equation}
	In the next section, we prove that~\eqref{eq:h1_a_priori} holds true.
	
	\section{Error analysis}\label{sec:error}
	
	We present the error estimate in energy norm for the source Problem~\ref{pb:source_rb}. As we discussed, this ensures the correct spectral approximation of the eigenpairs. We also prove the error estimate in $L^2-$norm. 
	
	Since the \textsf{rb}VEM approximation of Problem~\ref{pb:source_rb} is fully conforming, Galerkin orthogonality and Ce\`a's lemma hold straightforwardly~\cite{ciarlet}, yielding
	\begin{equation}\label{eq:error}
		| u-\uhrb |_{1,\Omega} \le C\, |u-\uIrb|_{1,\Omega}.
	\end{equation}
	We estimate the interpolation error in the following proposition.
	
	\begin{prop}\label{prop:interpolation_error}
		Let $\u\in\Hr{1+s}{\Omega}$, $1/2<s\le1$, be the unique solution to Problem~\ref{pb:source_cont} and $\uIrb\in\Vhrb$ its interpolant. There exists a constant $C$, independent of $h$, such that the following estimate holds
		$$
		| u-\uIrb |_{1,\Omega} \le C h^{{\red s}}\|f\|_{0,\Omega}.
		$$
	\end{prop}
	{\red
		\begin{proof}
			As usual, we estimate the interpolation error locally in a generic element $\element\in\mesh$. Thus, we consider the restriction of $\uIrb$ to $\element$ and denote it by $\interpolant u$. More precisely, the local interpolation operator $\interpolant:C^0(\element)\cap\HunoK\rightarrow\Vhrb(\element)\subset\HunoK$ is defined as
			\begin{equation}
				\interpolant u := \sum_{i=1}^{N} u(\vertex_i)\,\blue\baseVrb{i}\black,
			\end{equation}
			where $u(\vertex_i)$ are the degrees of freedom of $u$ in the discrete space $\Vhrb(\element)$, i.e. its values at the vertices of~$\element$, and $\blue\baseVrb{i}\black$ are the basis functions of the local space $\Vhrb(\element)$. 
			
			For $p \in \PunoK$ arbitrary we can write
			\begin{equation}\label{eq:tri_int}
				|u-\interpolant u|_{1,\element}
				\le
				|u-p |_{1,\element} + |p -\interpolant u|_{1,\element} 
			\end{equation}
			On the other hand, since $\interpolant p = p$, we estimate the term $|p -\interpolant u|_{1,\element}$ as
			\begin{equation}\label{eq:term_op}
				| p -\interpolant u|_{1,\element}
				= |\interpolant p-\interpolant u|_{1,\element}
				= |\interpolant(p - u)|_{1,\element}.
			\end{equation}
			\newcommand{\ballE}{B_E}
			Now we can write
			\begin{equation}\label{eq:bound}
				|\interpolant (u - p)|_{1,\element}
				= \left| \sum_{i=1}^{N} (u(\vertex_i) - p(\vertex_i))\,\blue\baseVrb{i}\black \right|_{1,\element}
				\le \left(\sum_{i=1}^{N}|\blue\baseVrb{i}\black|_{1,\element}\right) \|u - p\|_{\infty,\element}.
			\end{equation}
			Now, thanks to our shape regularity assumptions and a standard scaling argument \rv{applied ``piecewise'' to each cell of the shape-regular sub-triangulation of $\element$ and $\refelement$ (see Figure~\ref{fig:affine_mapping})}\black, we have that
			\[
			|\blue\baseVrb{i}|_{1,\element}  \simeq 1.
			\]
			Then, as the shape regularity of the elements implies that $N$ is uniformly bounded,  we have
			\begin{equation}
				| u - \interpolant u |_{1,\element} \le C\, \inf_{p \in \PunoK} \left(
				| u - p |_{1,\element} + \| u - p \|_{\infty,\element} 	\right).
			\end{equation}
			We can now take $p$ as the $L^2(\ballE)$ projection of an $H^s$ bounded extension of $u$ to the smallest ball $\ballE$ circumscribed to $\element$, and by \cite[Theorem 3.1.4]{ciarlet} we obtain
			\begin{equation}
				| u - \interpolant u |_{1,\element} \le C\, h^s | u |_{1+s,\element}.
			\end{equation}
			The global estimate is then obtained by summing on all local contributions.
		\end{proof}
	}
	
	By combining~\eqref{eq:error} with Proposition~\ref{prop:interpolation_error}, we obtain the following optimal error estimate in energy norm.
	
	\begin{teo}\label{teo:h1}
		Let $\u\in\Hr{1+s}{\Omega}$, $1/2<s\le1$, be the solution to Problem~\ref{pb:source_cont} and let $\uhrb\in\Vhrb$ be the solution to the discrete Problem~\ref{pb:source_rb}. There exists a positive constant $C$, independent of $h$, such that
		$$
		| u - \uhrb |_{1,\Omega} \le C h^{{\red s}} \|f\|_{0,\Omega}.
		$$
	\end{teo}
	
	By Poincar\'e inequality in $\Hunozero$, we exploit the equivalence between $\|\cdot\|_{1,\Omega}$ and $|\cdot|_{1,\Omega}$ to get the uniform convergence~\eqref{eq:unif_conv} directly from the a priori error estimate.
	
	In order to complete the error analysis for the source problem, we resort to a standard duality argument for proving the error estimate in $L^2-$norm.
	
	\begin{teo}\label{theo:l2_rbvem}
		Let $\u\in\Hr{1+s}{\Omega}$, $1/2<s\le1$, be the solution to Problem~\ref{pb:source_cont} and let $\uhrb\in\Vhrb$ be the solution to the discrete Problem~\ref{pb:source_rb}. There exists a positive constant $C$, independent of $h$, such that, {\red letting $\sigma_0 = \pi/\theta_0$ where $\theta_0$ is the maximum of the interior angles  of the polygonal domain $\Omega$, for $\epsilon > 0$ arbitrarily small we have}
		$$
		\| u - \uhrb \|_{0,\Omega} \le C h^{{\red s}+{\red \sigma_0-\epsilon}} \|f\|_{0,\Omega}.
		$$
	\end{teo}
	
	\begin{proof}
		Let $\udual\in\Hr{1+\sigma}{\Omega}$ be the solution to the dual problem
		\begin{equation}
			a(v,\udual) = (v,u-\uhrb)\qquad\forall v\in\Hunozero,
		\end{equation}
		which satisfies, for all $\sigma = \sigma_0-\epsilon$ with $ 0 < \sigma < \sigma_0$ (see \cite{grisvard2011elliptic}), 
		\begin{equation}\label{eq:stability_eta}
			|\eta|_{1+\sigma,\Omega} \le C \|u-\uhrb\|_{0,\Omega}.
		\end{equation}
		By testing the above equation with $v=u-\uhrb$ we estimate the error in the $L^2-$norm as follows
		\begin{equation}
			\begin{aligned}
				\|u-\uhrb\|_{0,\Omega}^2 &= a(u-\uhrb,\udual) = a(u-\uhrb,\udual-\udual_{I})&&\text{(Galerkin orthogonality)}\\
				& \le C\, |u-\uhrb|_{1,\Omega} |\udual-\udual_{I}|_{1,\Omega}&&\text{(continuity of $a$)}\\
				& \le C h^{{\red s}}|u|_{1+s,\Omega}|\udual-\udual_{I}|_{1,\Omega}&&\text{(Theorem~\ref{teo:h1})}\\
				& \le C h^{{\red s}}\|f\|_{0,\Omega}\,h^{{\red \sigma}}\,|\udual|_{1+\sigma,\Omega}&&\text{(stability and interpolation estimate)}\\
				&\le C h^{{\red s}+{\red \sigma}}\|f\|_{0,\Omega}\|u-\uhrb\|_{0,\Omega},&&\text{(\eqref{eq:stability_eta} and $1/2<s\le1$)}\\
			\end{aligned}
		\end{equation}
		where $\eta_I\in\Vhrb$ denotes the interpolant of $\eta$ in the discrete space $\Vhrb$.
		A division by $\|u-\uhrb\|_{0,\Omega}$ yields the desired result. 
	\end{proof}
	
	{\red	\begin{rem}
			We observe that in the case if the domain $\Omega$ is convex, we gain an additional power of $h$ in the $L^2-$norm estimate with respect to the error in energy norm. On the other hand, if the domain has large interior angles we gain a power $h^{\sigma_0 - \varepsilon}$, $\sigma_0$ depending on the maximum interior angle of the domain.
	\end{rem}}
	
	\section{Error estimates for eigenvalue problem}\label{sec:error_eig}
	
	We are now ready to present the convergence results for the approximation of the Laplace eigenproblem in \textsf{rb}VEM framework.
	
	By applying the results given by the Babu\v{s}ka--Osborn theory (see~\cite[Thm.~9.3--9.7]{boffi-acta} and~\cite[Thm.~7.1--7.4]{babuska-osborn}), we obtain the error estimates for the approximation of both the eigenvalues and eigenfunctions. To this aim, we recall the definition of gap between two spaces $\Ecal\subset\Hr{\ell}{\Omega}$ and $\Fcal\subset\Hr{\ell}{\Omega}$ as
	$$
	\deltahat_\ell(\Ecal,\Fcal) = \max\{\delta_\ell(\Ecal,\Fcal),\delta_\ell(\Fcal,\Ecal)\},
	$$
	where
	$$
	\delta_\ell(\Ecal,\Fcal) = \sup_{v\in\Ecal,\|v\|_{\ell,\Omega}=1}\inf_{w\in\Fcal,\|w\|_{\ell,\Omega}=1}\|v-w\|_{\ell,\Omega}.
	$$
	In our case, we are interested in $\ell=0,1$ measuring the gap either in the $L^2-$ or $H^1-$ norm.
	
	\begin{teo}\label{theo:convergence}
		Let $\eig$ be an eigenvalue of Problem~\ref{pb:weak} of multiplicity $m$ and $\Ecal_\eig$ be the corresponding eigenspace. Then, there are exactly $m$ discrete eigenvalues of Problem~\ref{pb:rb_eig} $\eig_{j(i),h}^{\blue\mathrm{rb}\black}$ ($i=1,\dots,m$) tending to $\eig$. Moreover, assuming that all $\eigf\in\Ecal_\eig$ satisfy $\eigf\in\Hr{1+s}{\Omega}$ with $1/2<s\le1$, the following inequalities hold true
		\begin{equation*}
			|\eig-\eig_{j(i),h}^{\blue\mathrm{rb}\black}|\le Ch^{\red 2s},\qquad
			\deltahat_1(\Ecal_\eig,\oplus_i\Ecalrb_{j(i),h})\le C h^{\red s},
			\qquad\deltahat_0(\Ecal_\eig,\oplus_i\Ecalrb_{j(i),h})\le C h^{\red s + \sigma_0-\epsilon},
		\end{equation*}
		where  $\Ecalrb_{r,h}$ is the eigenspace spanned by the {\red approximate} eigenfunctions $\eigfrb_{r,h}$.
	\end{teo}
	
	{\red Theorem \ref{theo:convergence} gives us an approximation result for the eigenvalues and for the eigenfunctions in $\Vhrb$.	If we interpret Problem \ref{pb:rb_eig} as a non conforming discretization of Problem \ref{pb:weak} in the global VEM space $\Vh$, the eigenfunctions will be reconstructed using the basis for $\Vh$ rather than the one for $\Vhrb$, so that such a theorem will not apply directly. In order to get a convergence estimate we will resort to the following lemma.

		\begin{lem}\label{lem:transfer}
			Let $\vhrb \in \Vhrb(\element)$ and let $v_h \in \vunoloc$ be defined by $v_h = \sum_i \vhrb(\vertex_i) \base_{i}$. Then it holds
			\begin{gather}
				| \vhrb - v_h |_{1,\element} + h^{-1} 	\| \vhrb - v_h \|_{0,\element} \le C\, \inf_{p \in \PunoK} \left(| \vhrb - p |_{1,\element} + h^{-1} \| \vhrb - p \|_{0,\element}
				\right)
			\end{gather}
		\end{lem}
		\begin{proof}
			
			We observe that, as both $\vunoloc$ and $\Vhrb(E)$ exactly reproduce linears, for all $p \in \PunoK$ we have that
			\[
			\sum_i p(\vertex_i)(\base_i - \blue\baseVrb{i}) = 0.
			\]
			Then for $p \in \PunoK$ arbitrary we can write
			\begin{equation}
				\begin{aligned}
					| \vhrb - v_h |_{1,\element}  &= 	| \sum_i \vhrb(\vertex_i)  (\blue\baseVrb{i}
					- \base_i) |_{1,\element} \\&= | \sum_i (\vhrb(\vertex_i) - p(\vertex_i) ) (\blue\baseVrb{i}
					- \base_i) |_{1,\element}\\
					&\leq  \| \vhrb - p \|_{\infty,\partial\element} \sum_i( | \blue\baseVrb{i} |_{1,\element} +	| \base_i |_{1,\element}) \le C\,  \| \vhrb - p \|_{\infty,\partial\element}.
				\end{aligned}
			\end{equation}
			By a standard scaling argument, as $\vhrb - p$ is a piecelinear function on $\partial \element$ we obtain
			\begin{equation}	| \vhrb - v_h |_{1,\element}  \le C\,\left( | \vhrb - p |_{1/2,\partial\element} + h^{-1/2} \| \vhrb - p \|_{0,\partial\element}\right).
			\end{equation}
			By  a standard trace theorem we then have
			\begin{equation}
				\begin{aligned}
					| \vhrb - v_h |_{1,\element}  &\le C\,\left( | \vhrb - p |_{1,\element} + h^{-1/2} \| \vhrb - p \|_{0,\element}^{1/2} | \vhrb - p |_{1,\element}^{1/2}\right)
					\\ &\le C\,\left( | \vhrb - p |_{1,\element} + h^{-1} \| \vhrb - p \|_{0,\element}\right).
				\end{aligned}
			\end{equation}
			By the arbitrariness of $p$ we get the desired result.
		\end{proof}
		Combining Theorem \ref{theo:convergence} with Lemma \ref{lem:transfer} we obtain the followin corollary
		\begin{corollary} Under the assumptions of Theorem \ref{theo:convergence}, letting 
			$\Ecal_{r,h} \subset \Vh$ be the eigenspace spanned by the approximate  eigenfunctions $\eigf_{r,h}$ of Problem \ref{pb:equiv_vem_eig}, we have that 
			\begin{equation*}
				\deltahat_1(\Ecal_\eig,\oplus_i\Ecal_{j(i),h})\le C h^t,
				\qquad\deltahat_0(\Ecal_\eig,\oplus_i\Ecal_{j(i),h})\le C h^{t+s},
			\end{equation*}
		\end{corollary}
		\begin{proof}
			Let $u_h\in \Vh$ and $\uhrb \in \Vhrb$ be such that for all nodes $\vertex$ of the mesh $\mesh$ we have $u_h(\vertex) = \uhrb(\vertex)$ and such that $u_h$ and $\uhrb$ are, respectively, functions in the eigespaces $\Ecal_{j(i),h}$ and $\Ecalrb_{j(i),h}$. By Theorem~\ref{theo:convergence}, there exists a $u \in \Ecal_\lambda$ with $u \in H^{s+1}(\Omega)$ such that 
			\[
			| u - \uhrb |_{1,\Omega} +h^{-\sigma_0+\varepsilon} 
			\| u - \uhrb \|_{0,\Omega}
			\le C h^s.
			\]
			Now, letting $\Pi^0_E$ denote the $L^2(E)$ projection onto $\PunoK$, we can write
			\begin{equation*}
				\aligned
				&\| \uhrb - u_h \|_{0,E} \leq
				| \uhrb - u_h |_{1,E} + h^{-1}	\| \uhrb - u_h \|_{0,E} \\
				&\qquad\leq		 | \uhrb - \Pi^0_\element u - \fint_\element (\uhrb - u) |_{1,E} + h^{-1} \| \uhrb - \Pi^0_\element u - \fint_\element (\uhrb - u) \|_{0,\element}\\ &\qquad\leq
				| \uhrb -  u |_{1,\element} + h^{-1} \| \uhrb -  u - \fint(\uhrb - u) \|_{0,\element} +  | u - \Pi^0 u |_{1,\element} + h^{-1} \| u - \Pi^0 u \|_{0,\Omega} \\
				&\qquad\le C\,\left( 	 | \uhrb -  u |_{1,\element} +  | u - \Pi^0 u |_{1,\element} + h^{-1} \| u - \Pi^0 u \|_{0,\element}\right).
				\endaligned
			\end{equation*}
			
			Theorem \ref{theo:convergence} combined with standard piecewise polynomial approximation results immediately yield
			\[\delta_1(\oplus_i\Ecal_{j(i),h},\Ecal_\lambda) \leq C h^s, \qquad \delta_0(\oplus_i\Ecal_{j(i),h},\Ecal_\lambda) \leq C h^{s+\sigma_0-\epsilon}.\]
			Analogously, one can prove that
			\[\delta_1(\Ecal_\lambda,\oplus_i\Ecal_{j(i),h}) \leq C h^s, \qquad \delta_0(\Ecal_\lambda,\oplus_i\Ecal_{j(i),h}) \leq C h^{s+\sigma_0-\epsilon},\]
			which concludes the proof.
		\end{proof}
	}
	
	\section{Numerical tests}\label{sec:tests}
	
	In this section we present a wide set of numerical tests confirming our theoretical findings. We first consider the Laplace eigenproblem on the unit square. We investigate the robustness of the \textsf{rb}VEM with respect to the number~$M$ of reduced basis functions employed to construct the discrete space. Then, we compare the approximation properties of our novel \textsf{rb}VEM scheme with those of standard VEM and the alternative rb--stab--VEM discussed in Remark~\ref{rem:rbstab}. We further test the optimality our method in the case of singular eigenfunctions by considering the L-shaped domain benchmark test and, finally, we solve a second order problem with piecewise constant coefficient.
	
	We point out that the discrete space $\Vhrb(\element)$ for a generic mesh element $\element$ with $N$ vertices is constructed by means of the reduced basis functions computed once and for all as described in~\cite[Sect. 6]{credali}. \rv More precisely, the offline construction of the reduced basis machinery was carried out for polygons having $N=4,\dots,14$ edges and on fine meshes $\meshdelta$ with $\delta=0.01$. Regarding the online construction of the ``virtual'' component of the stiffness and mass matrices, we remark that the computational cost is of the same order of magnitude of the cost of the evaluation of the consistency part. Thus, \textsf{rb}VEM is {only moderately} more expensive than the standard stabilized VEM, but generally leads to convergence improvements (see~\cite[Sect. 7]{credali}).\black
	
	\subsection{Laplace eigenproblem on the unit square}
	
	In order to have at disposal the eigenpairs in analytical form, we solve the Laplace eigenvalue problem in the unit square. In this case, the exact eigenvalues are given by $(i^2+j^2)\pi^2$ with $i,j\in\mathbb{N}\setminus\{0\}$ with eigenfunctions $\sin(i\pi x)\sin(j\pi y)$.
	
	As we discussed in Section~\ref{sec:rbvem}, the definition of the new discrete space $\Vhrb$ depends on the number~$M$ of reduced basis functions used to approximate the virtual basis functions, see~\eqref{eq:wrb} and~\ref{eq:rbspace}. On the other hand, we point out that the number~$M$ is fixed \textit{a priori}. The careful reader may wonder whether the computed eigenvalues depend on this parameter $M$. In order to investigate this aspect, we solve Problem~\ref{pb:rb_eig} with different discrete spaces $\Vhrb$ for several choices of $M$, varying from 1 to~50. We also consider three different meshes: Voronoi, dyadic (i.e. squares with an additional dofs at each edge middle point), and octagons, as depicted in Figure~\ref{fig:M-ind}. In the same figure, we also plot the discrete eigenvalues rescaled by a factor $\pi^2$ with respect to $M$. It is evident that the computed eigenvalues lie on horizontal lines, meaning that our method is robust since the eigenvalues are independent of $M$. For this reason, in the following numerical tests, we construct the space $\Vhrb$ with $M=1$.
	
	\begin{figure}[t]
		\subfloat[Voronoi mesh]{\includegraphics[width=0.3\linewidth,trim=30 20 20 0]{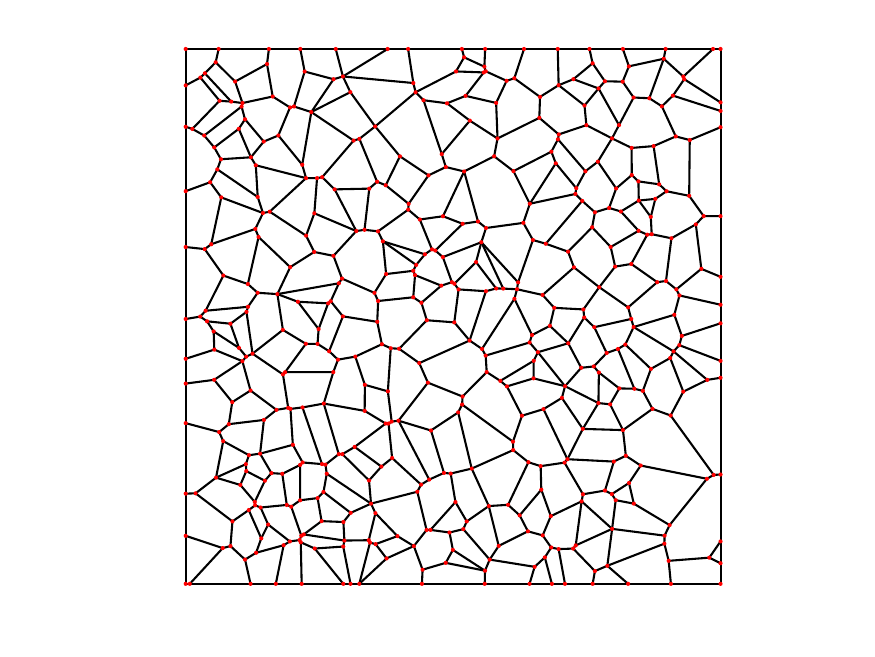}}\quad
		\subfloat[Dyadic mesh]{\includegraphics[width=0.3\linewidth,trim=20 20 20 0]{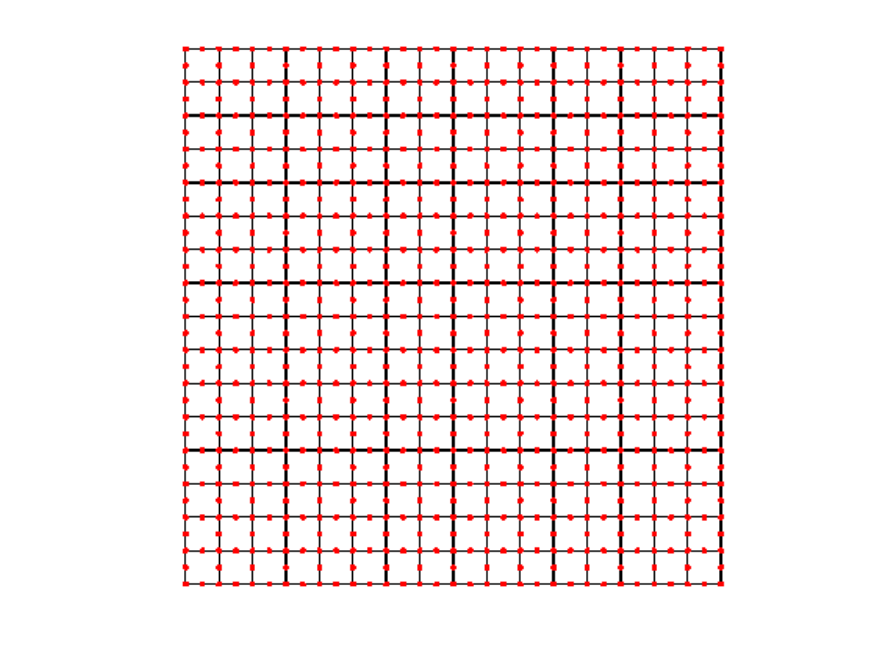}}\quad
		\subfloat[Octagons mesh]{\includegraphics[width=0.3\linewidth,trim=20 20 20 0]{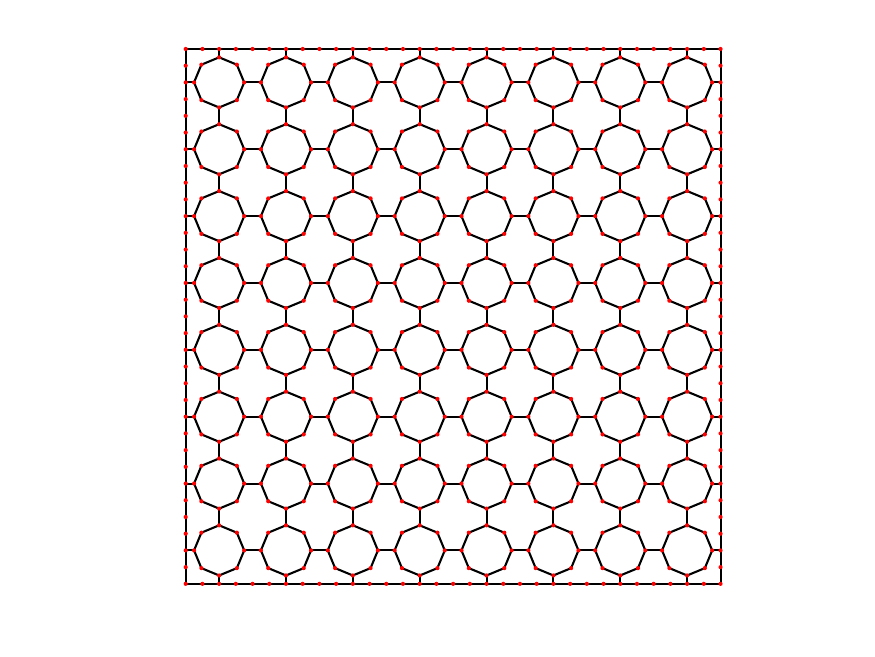}}
		
		\
		
		\subfloat[]{\includegraphics[width=0.325\linewidth]{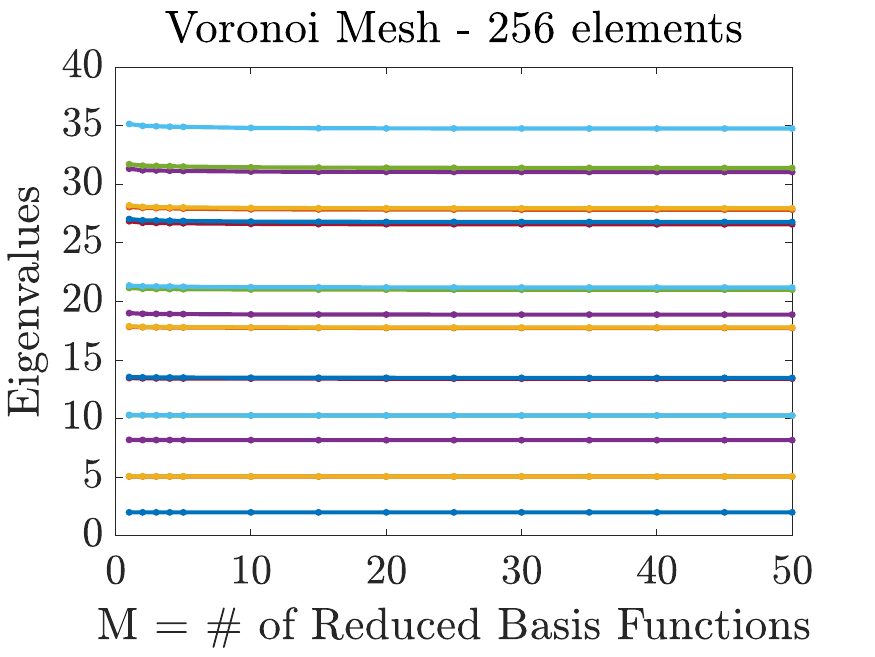}}
		\subfloat[]{\includegraphics[width=0.325\linewidth]{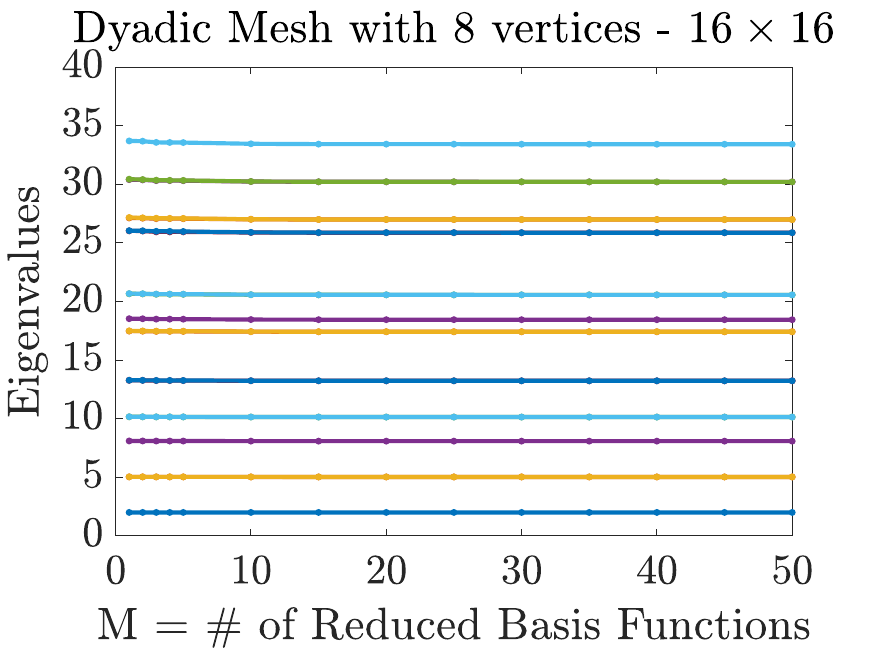}}
		\subfloat[]{\includegraphics[width=0.325\linewidth]{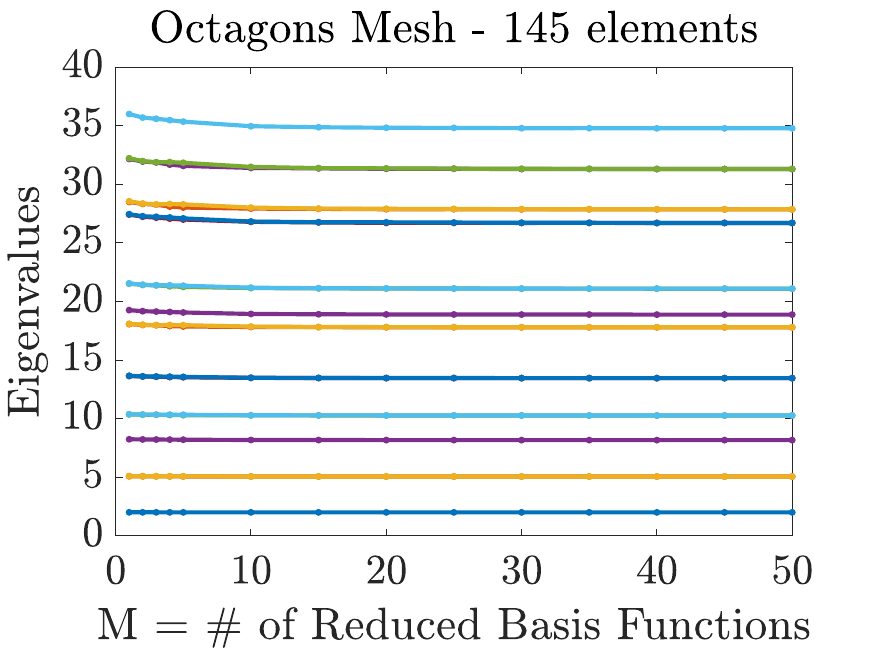}}
		
		\caption{The robustness of \textsf{rb}{VEM} with respect to the number of reduced basis functions $M$. In the top line we depict the three different considered meshes, whereas in the bottom line we plot the computed eigenvalues.}
		\label{fig:M-ind}
	\end{figure}
	
	We now compare the approximation properties of our \textsf{rb}{VEM} and rb--stab--VEM formulations with respect to those of the stabilized VEM on a sequence of Voronoi meshes. Concerning the VEM stabilization parameters, we consider two cases: a fully stabilized formulation with $\alpha=0.35$,  $\beta=1$  and a partially stabilized formulation with $\alpha=0.35$ and $\beta=0$, in the spirit of~\cite{BGG25}. The results are reported in Figure~\ref{fig:voronoi}, where each line refers to the employed numerical method, namely VEM ($\alpha=0.35$, $\beta=1$), VEM ($\alpha=0.35$, $\beta=0$), \textsf{rb}VEM and rb--stab--VEM ($\chi=0,1$), respectively. In the first column we plot the first 20 discrete and exact eigenvalues, in the second column we plot the error for the first 220 eigenvalues with respect to their index and, finally, in the third column we plot the error convergence for the first 10 eigenvalues when refining $h$. We observe that the standard VEM with $\beta=0$ and our \textsf{rb}VEM and rb--stab--VEM behave equivalently, showing a significant improvement with respect to standard VEM with $\beta=1$. We also point out that the convergence rate of the error is two as predicted by the theoretical results (see Theorem~\ref{theo:convergence}). Moreover, by comparing the error plots in the second column, we notice that the discrete methods based on the reduced basis approach approximate well a larger portion of the spectrum since less oscillations occur.In addition, the fully stabilized VEM ($\alpha=0.35$, $\beta=1$) produces spurious eigenvalues, as reported in Table~\ref{tab:spurious_eig_voronoi}, where spurious modes are displayed into boxes.
	
	\begin{figure}
		
		\includegraphics[width=0.325\linewidth]{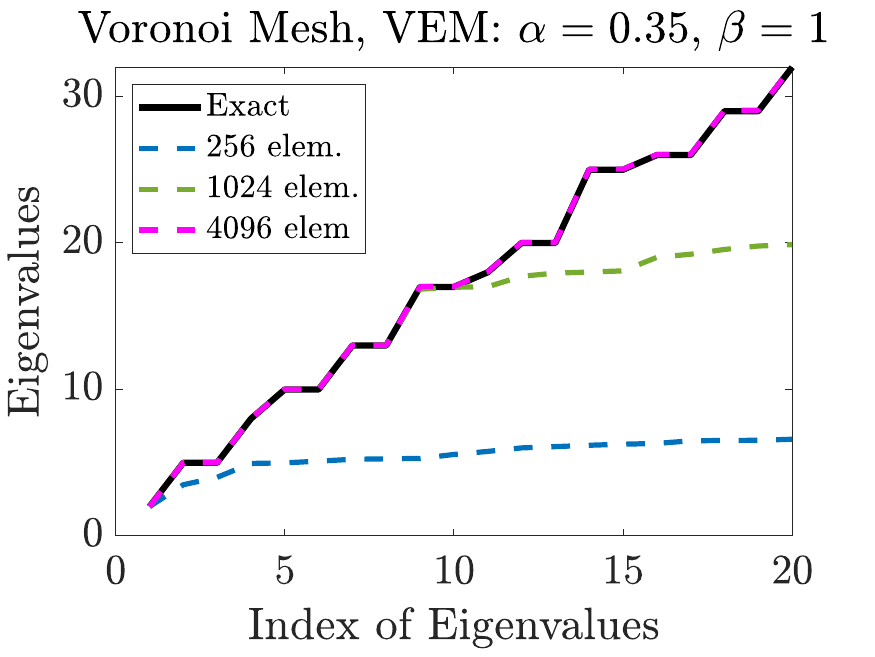}
		\includegraphics[width=0.325\linewidth]{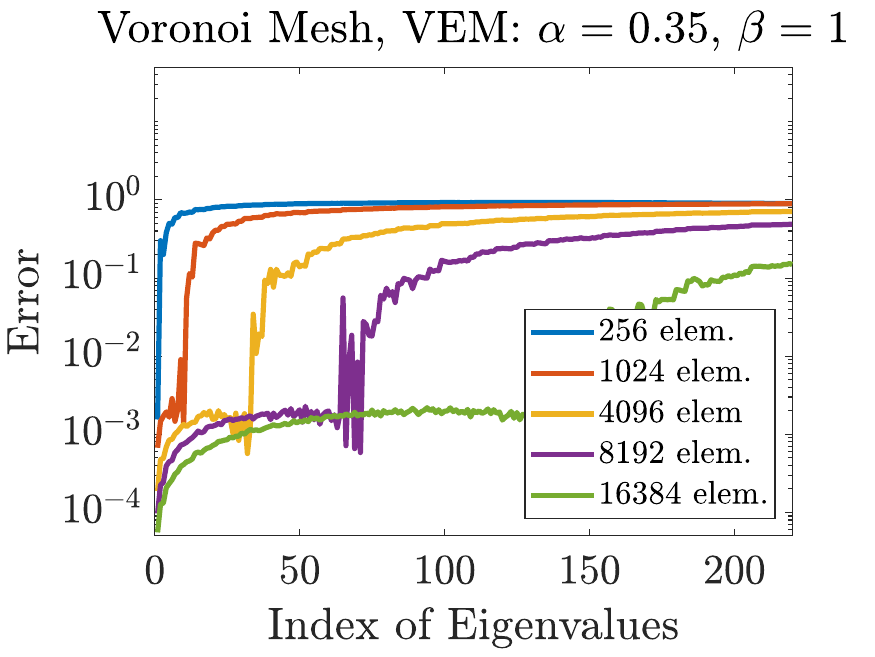}
		\includegraphics[width=0.325\linewidth]{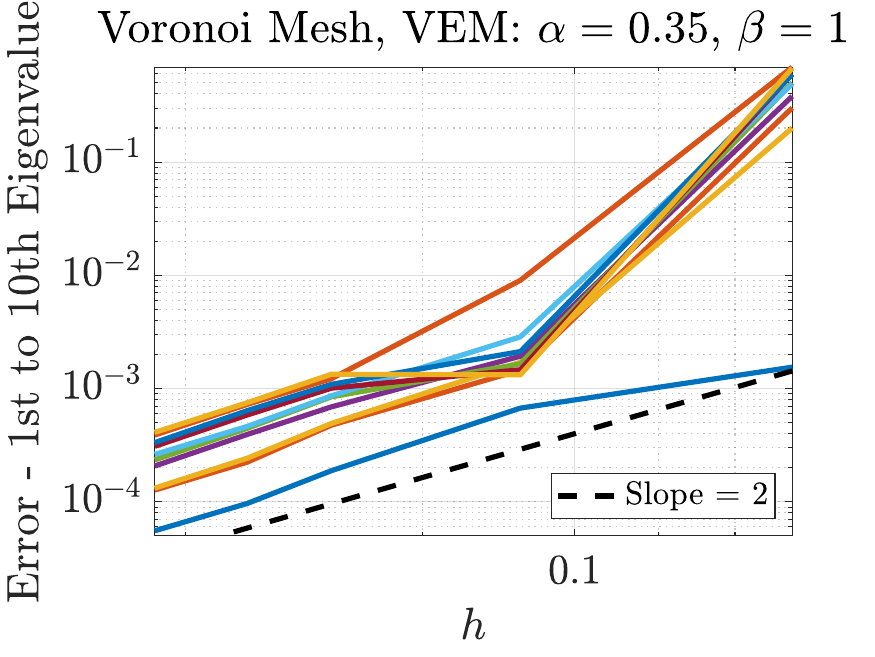}
		
		\
		
		\includegraphics[width=0.325\linewidth]{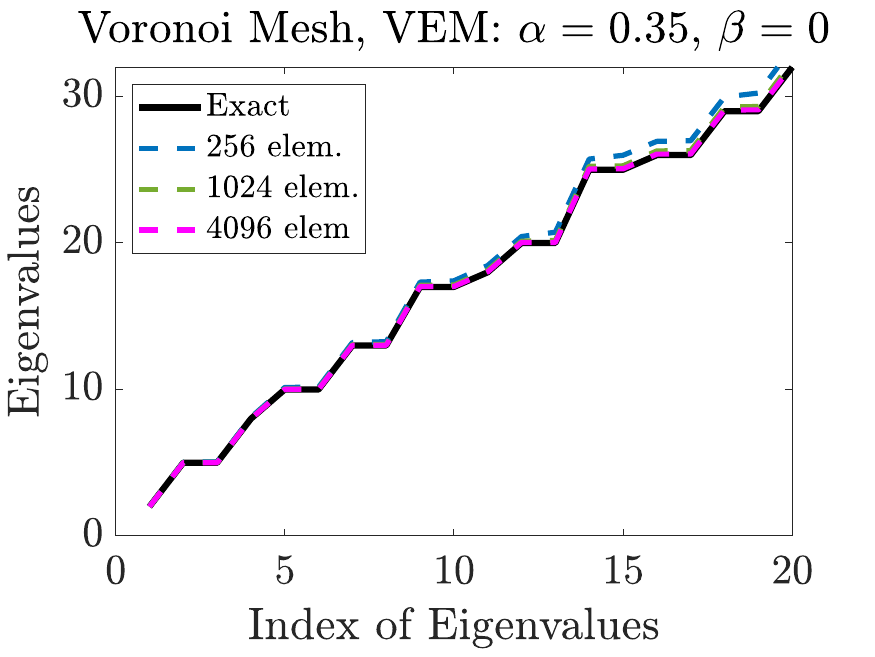}
		\includegraphics[width=0.325\linewidth]{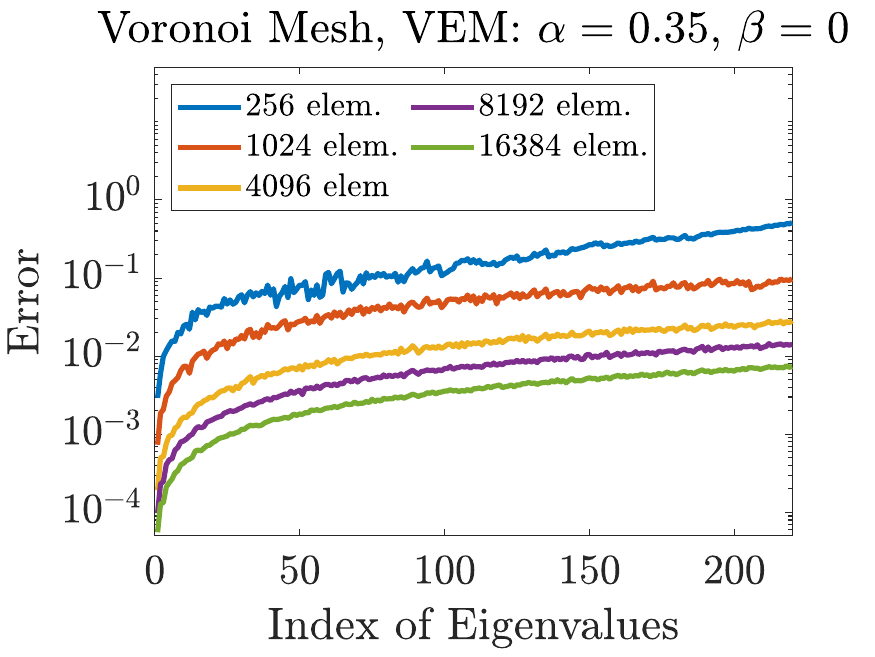}
		\includegraphics[width=0.325\linewidth]{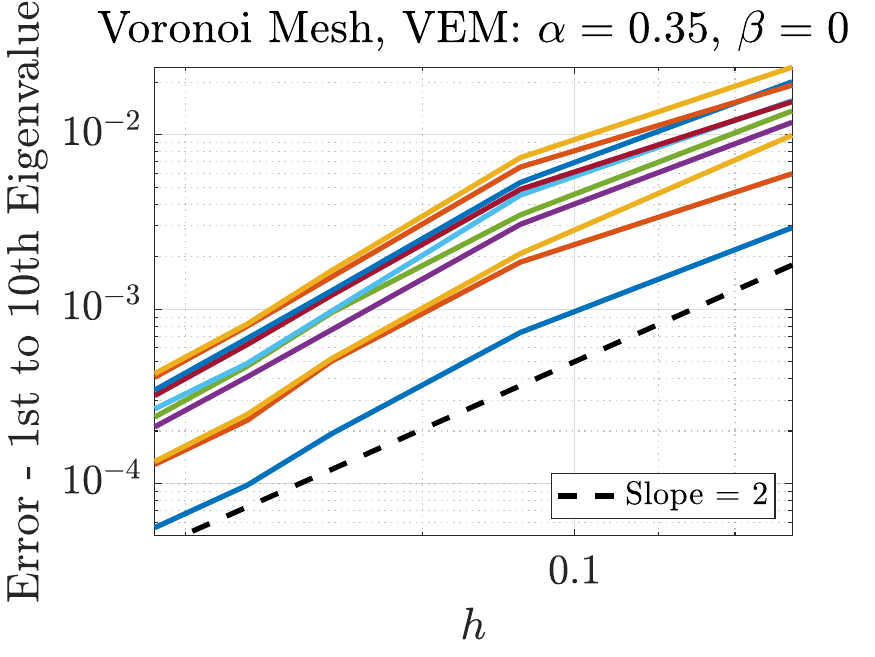}
		
		\
		
		\includegraphics[width=0.325\linewidth]{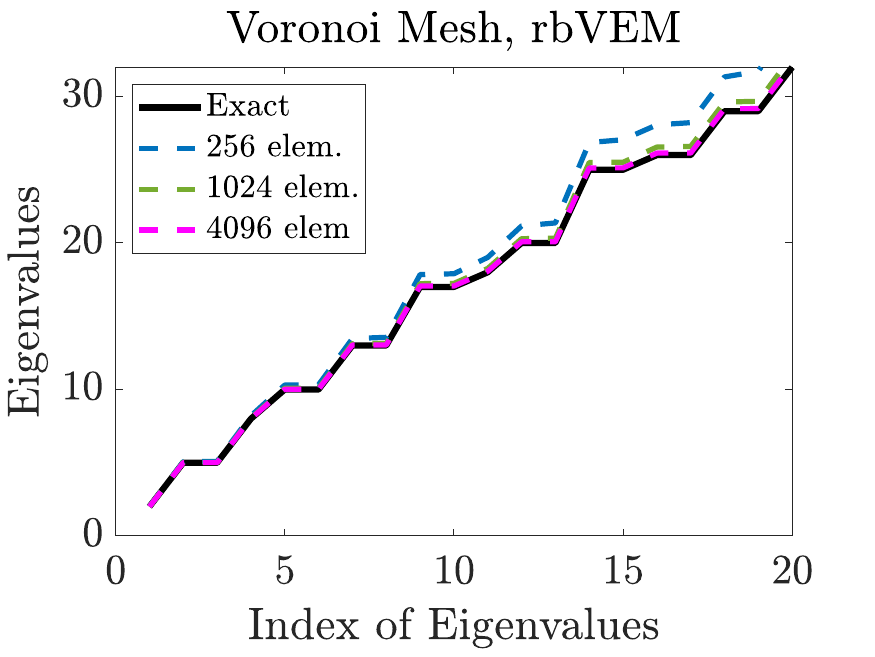}
		\includegraphics[width=0.325\linewidth]{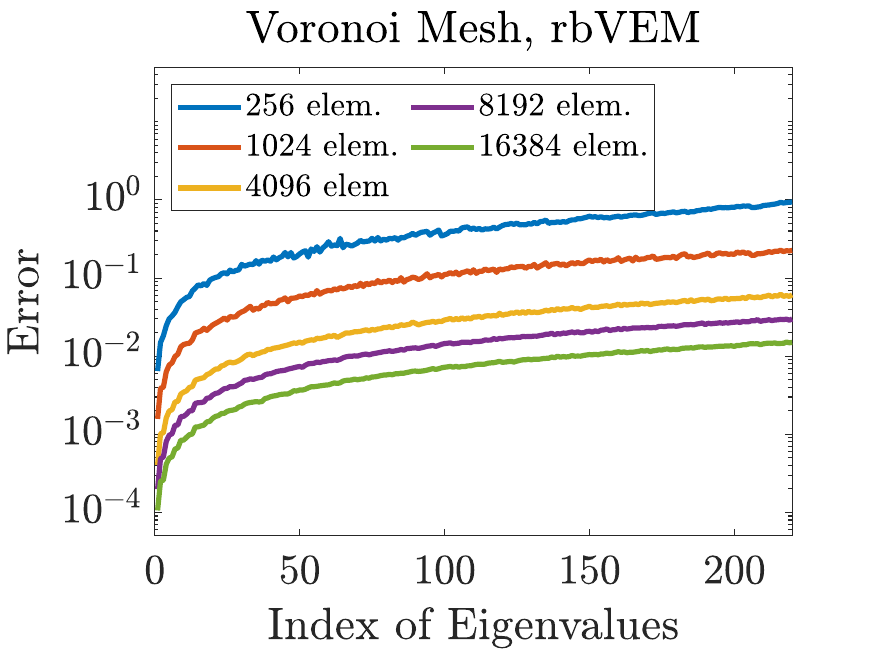}
		\includegraphics[width=0.325\linewidth]{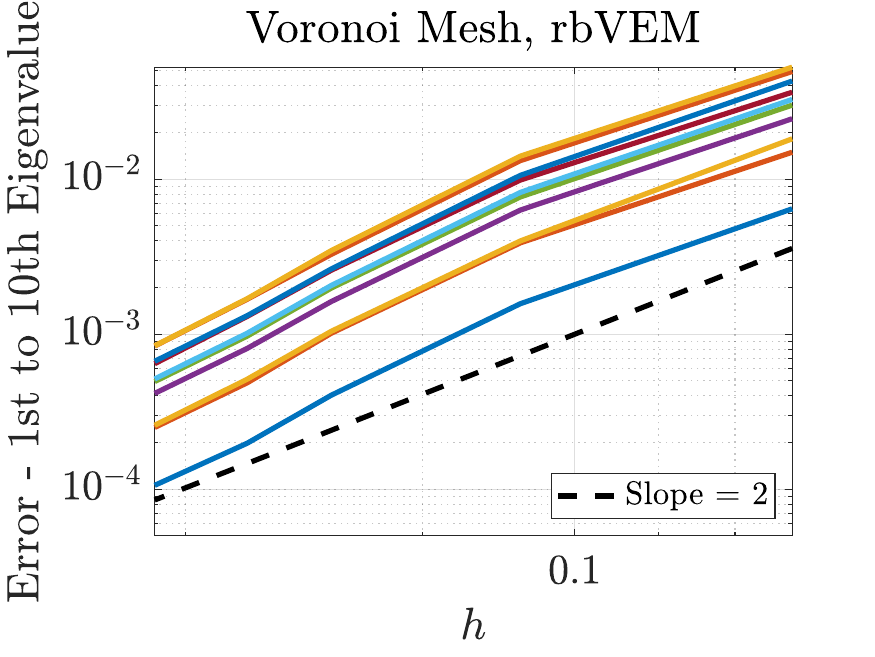}
		
		\
		
		\includegraphics[width=0.325\linewidth]{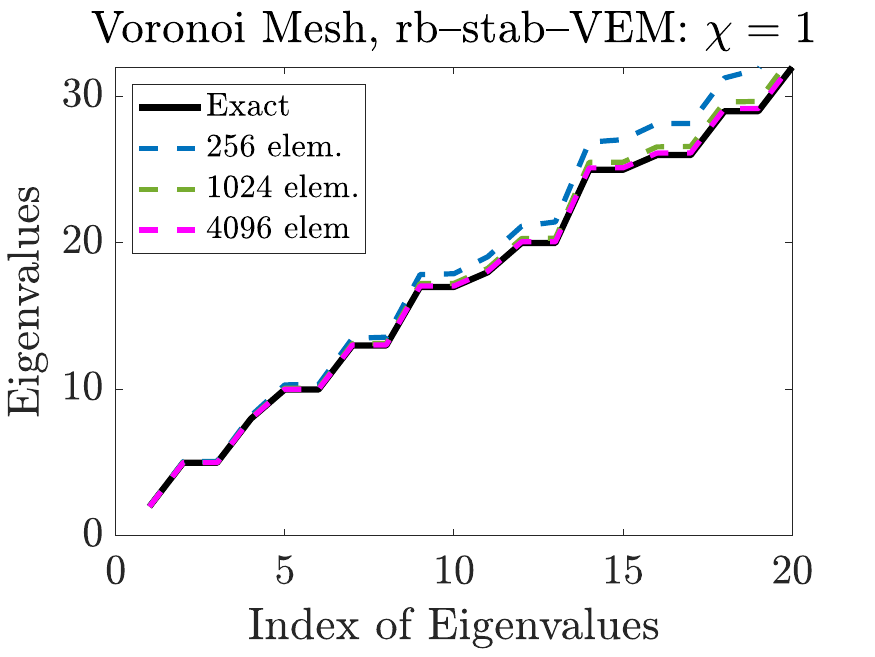}
		\includegraphics[width=0.325\linewidth]{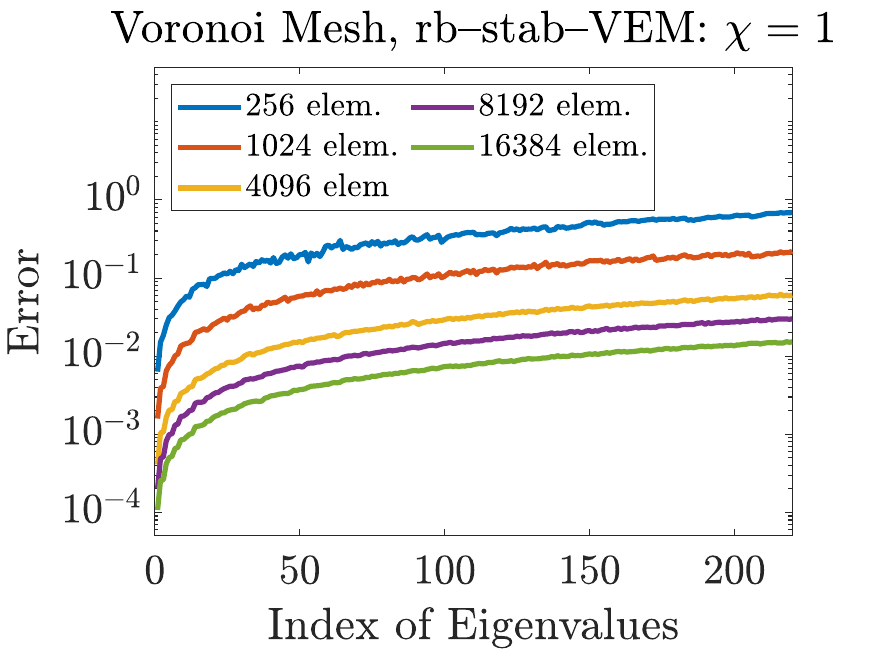}
		\includegraphics[width=0.325\linewidth]{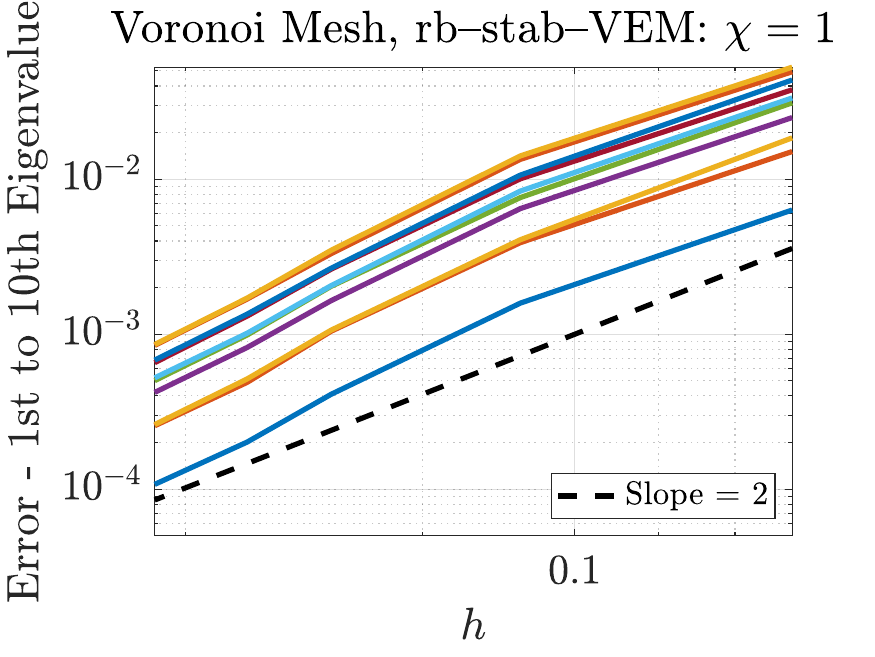}
		
		\
		
		\includegraphics[width=0.325\linewidth]{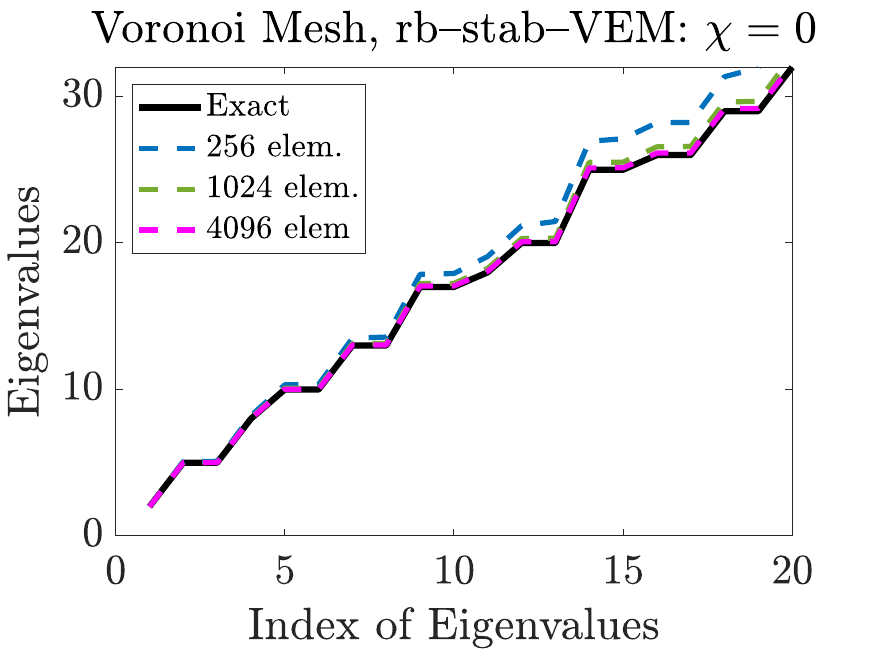}
		\includegraphics[width=0.325\linewidth]{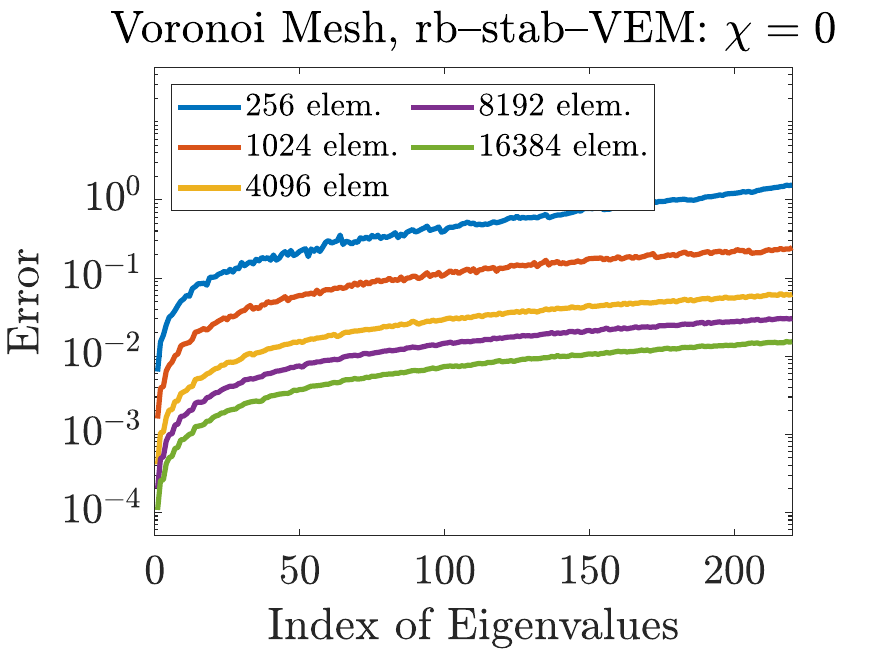}
		\includegraphics[width=0.325\linewidth]{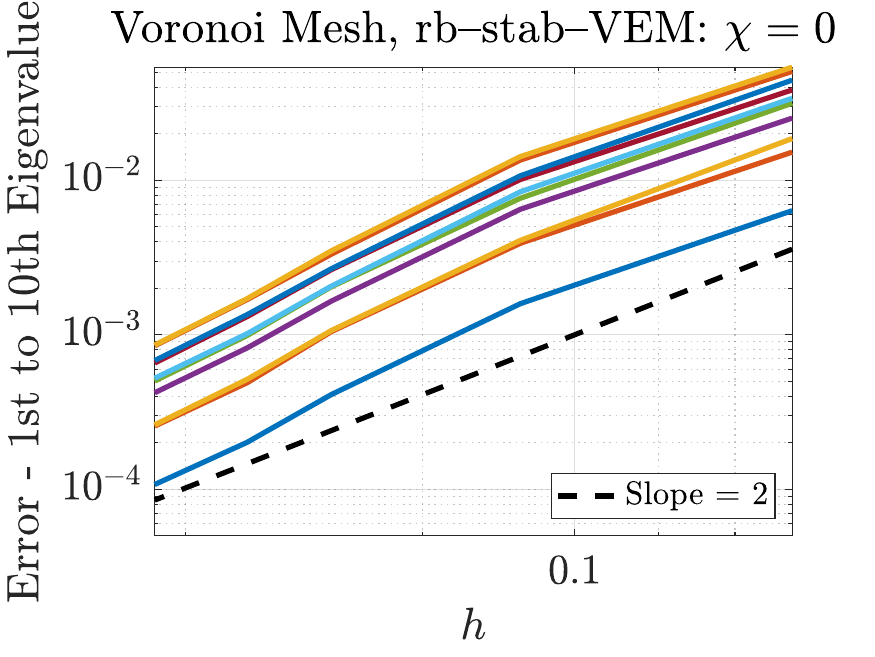}
		
		\caption{Comparison between VEM (first and second line) \textsf{rb}VEM (third line) and rb--stab--VEM (fourth and fifth line) on a sequence of Voronoi meshes. First column: plots of the first 20 computed and exact eigenvalues. Second column: error plots for the first 220 eigenvalues with respect to their index. Third column: error plots for the first 10 eigenvalues when refining the mesh size.}
		\label{fig:voronoi}
	\end{figure}
	
	\begin{table}
		\centering
		\renewcommand{\arraystretch}{1.5}
		\resizebox{\textwidth}{!}{
			\begin{tabular}{ccccccccccc}
				\multicolumn{11}{c}{\Large{\textbf{Eigenvalues on Voronoi Meshes - VEM $\alpha=0.35$, $\beta=1$}}}\\
				\hline
				\hline
				\textbf{Index of Eigenvalue} & \textbf{1} & \textbf{2} & \textbf{3} & \textbf{4} & \textbf{5} & &&&&\\
				256 Elements &2.0031&\framebox[1.1\width]{3.4950}&\framebox[1.1\width]{4.0009}&4.9486&4.9883&$\cdots$&&&&\\
				\hline
				\hline
				\textbf{Index of Eigenvalue} &  & \textbf{8} & \textbf{9} & \textbf{10} & \textbf{11} & \textbf{12} & \textbf{13} & \textbf{14} & \textbf{15} &\\
				1024 Elements & $\cdots$ & 13.0276 & \framebox[1.1\width]{16.8463} &  16.9775 & 17.0105	 & \framebox[1.1\width]{17.7225} & 17.9448	& \framebox[1.1\width]{18.0133} & 18.0961 & $\cdots$ \\
				\hline
				\hline
				\textbf{Index of Eigenvalue} &  & \textbf{30} & \textbf{31} & \textbf{32} & \textbf{33} & \textbf{34} & \textbf{35} & \textbf{36} & \textbf{37} &\\
				4096 Elements & $\cdots$ & 45.0668 & \framebox[1.1\width]{49.9084} &  50.0284 & 50.0628	 & \framebox[1.1\width]{50.2086} & \framebox[1.1\width]{51.4368}	& 51.9794 & 52.0531 & $\cdots$ \\
				\hline
				\hline
		\end{tabular}}
		\caption{Eigenvalues computed with Stabilized VEM on Voronoi meshes. Spurious eigenvalues are reported into boxes.}
		\label{tab:spurious_eig_voronoi}
	\end{table}\renewcommand{\arraystretch}{1}
	
	As discussed in Section~\ref{sec:vem1}, the term $b(\PiZero u_h,\PiZero v_h)$ may have a kernel, and this is the interesting case when considering VEM ($\beta=0$) or rb--stab--VEM ($\chi=0$) in the spirit of~\cite{BGG25}. Nevertheless, on the sequence of Voronoi meshes we considered for our tests, the non-stabilized mass term has full rank. Thus, we repeat the same comparison on a sequence of dyadic meshes, giving, for $\beta,\chi=0$, singular mass matrices. In this test, the stability parameters for standard VEM are set to $\alpha=\beta=1$. The associated results are collected in Figure~\ref{fig:dyadic}: we observe the same behavior as for the Voronoi case, so that we can draw the same conclusions.
	
	\begin{figure}
		
		\includegraphics[width=0.325\linewidth]{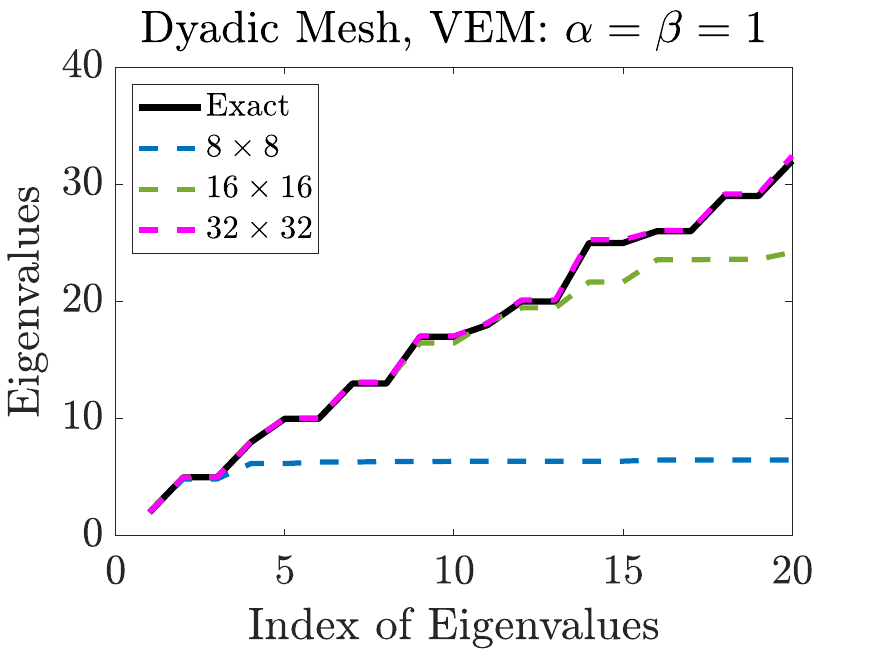}
		\includegraphics[width=0.325\linewidth]{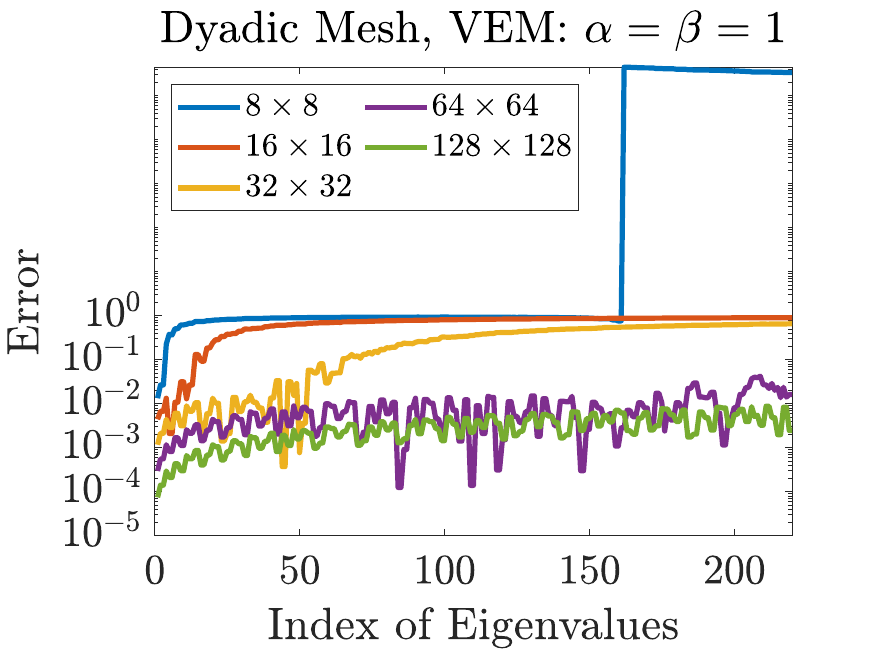}
		\includegraphics[width=0.325\linewidth]{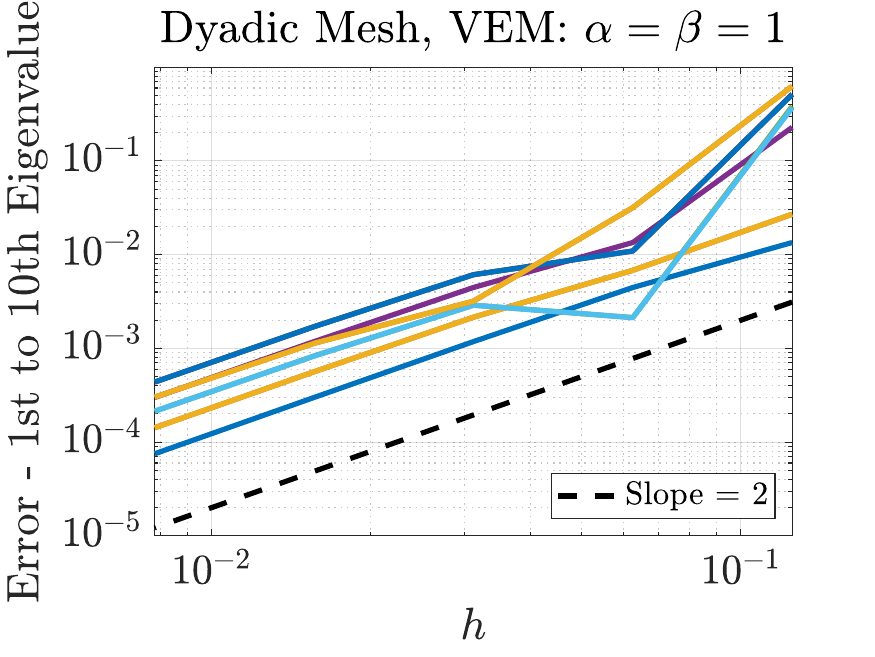}
		
		\

		\includegraphics[width=0.325\linewidth]{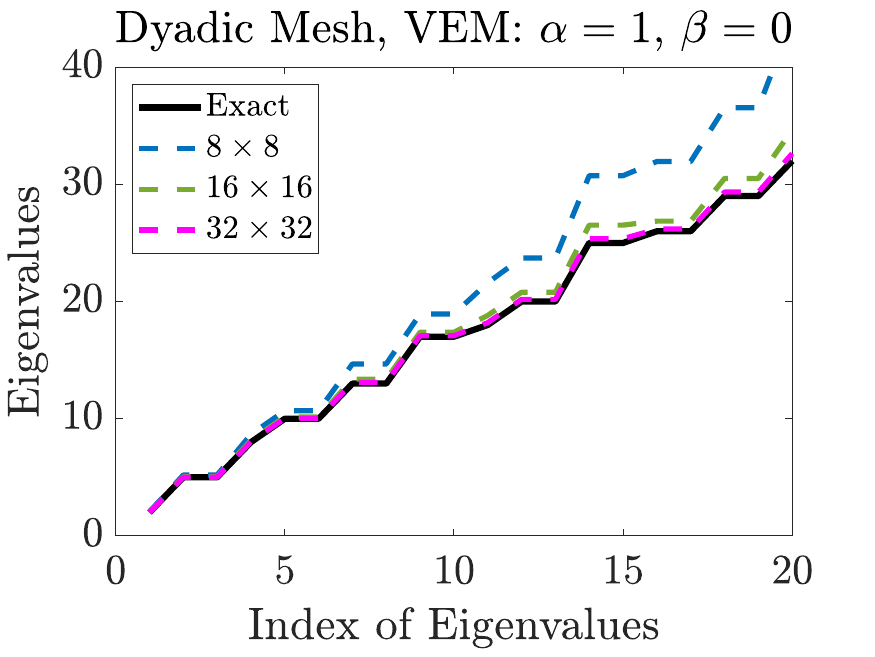}
		\includegraphics[width=0.325\linewidth]{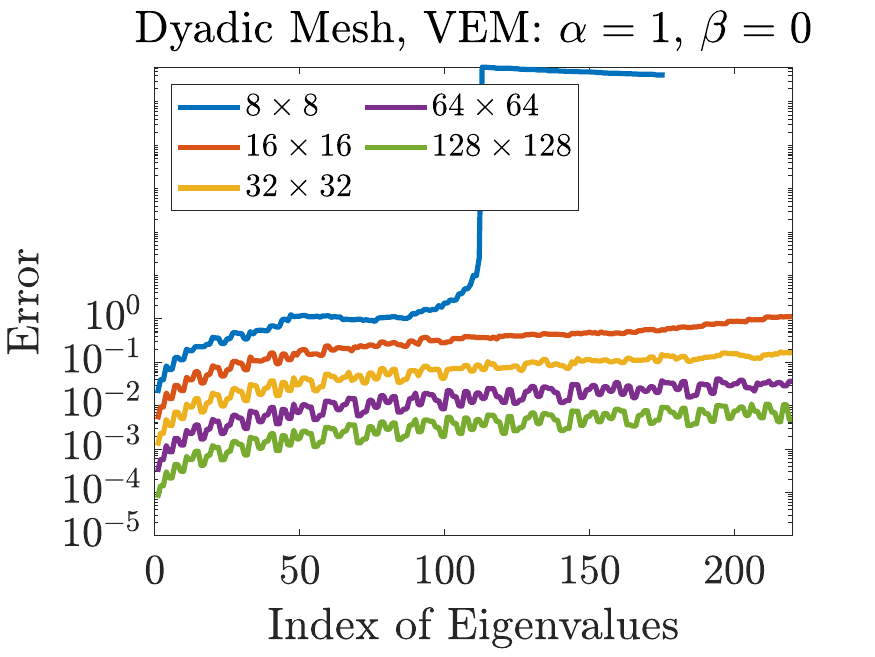}
		\includegraphics[width=0.325\linewidth]{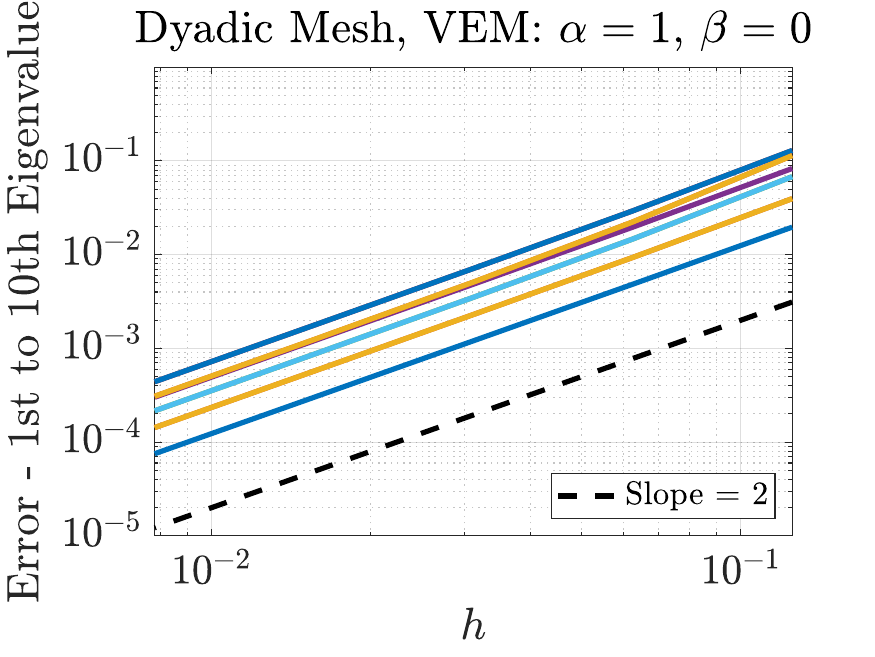}
		
		\
		
		\includegraphics[width=0.325\linewidth]{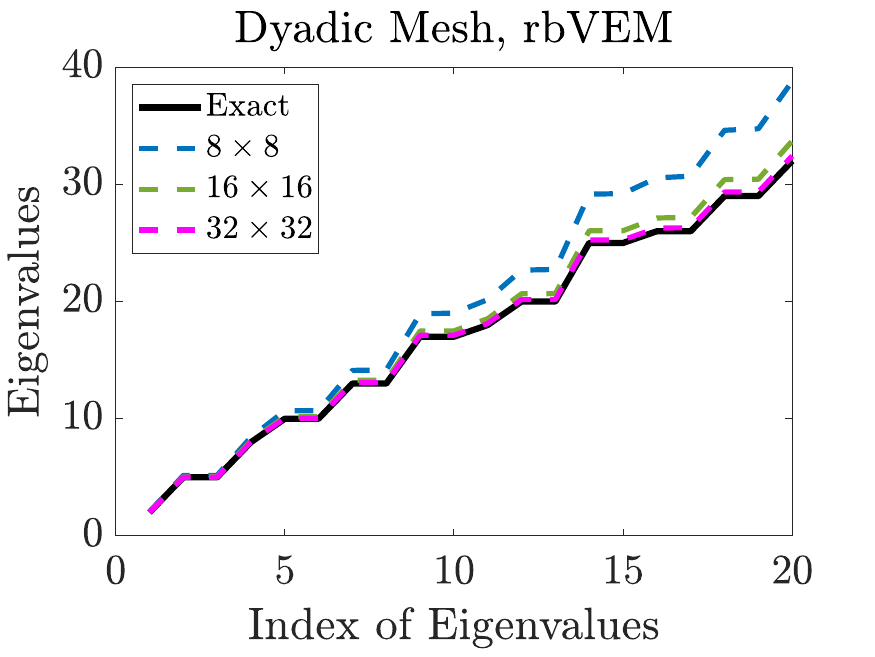}
		\includegraphics[width=0.325\linewidth]{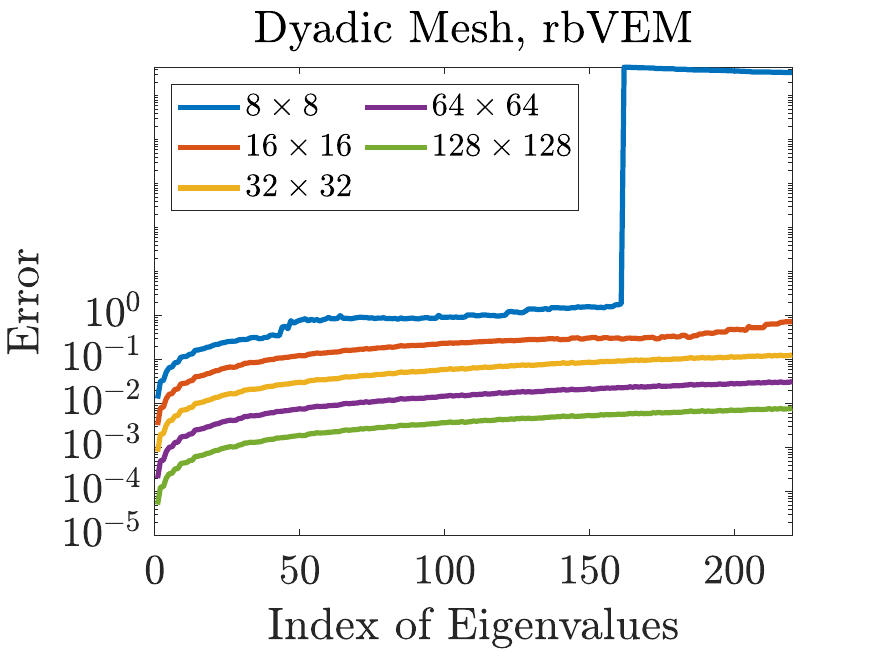}
		\includegraphics[width=0.325\linewidth]{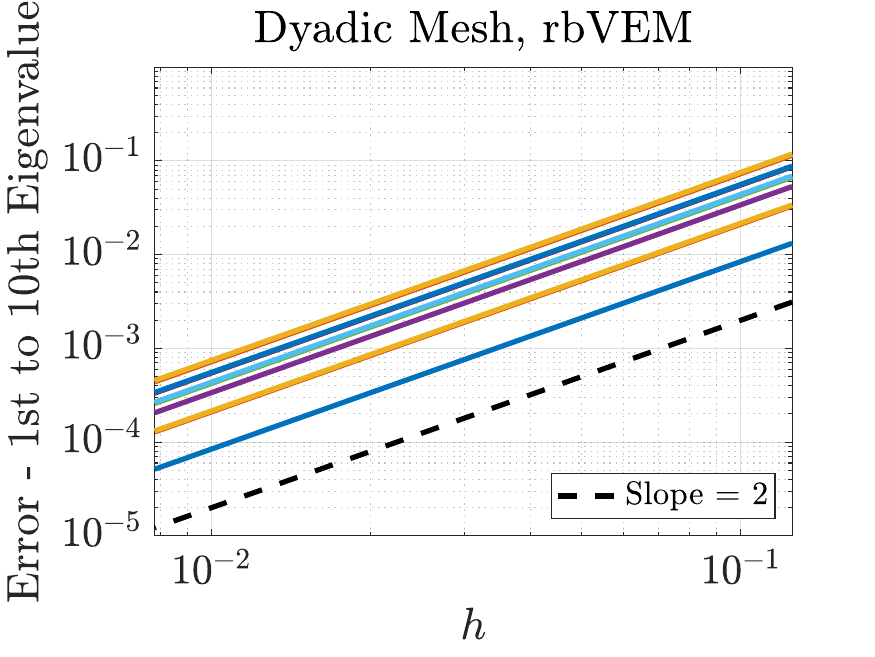}
		
		\
		
		\includegraphics[width=0.325\linewidth]{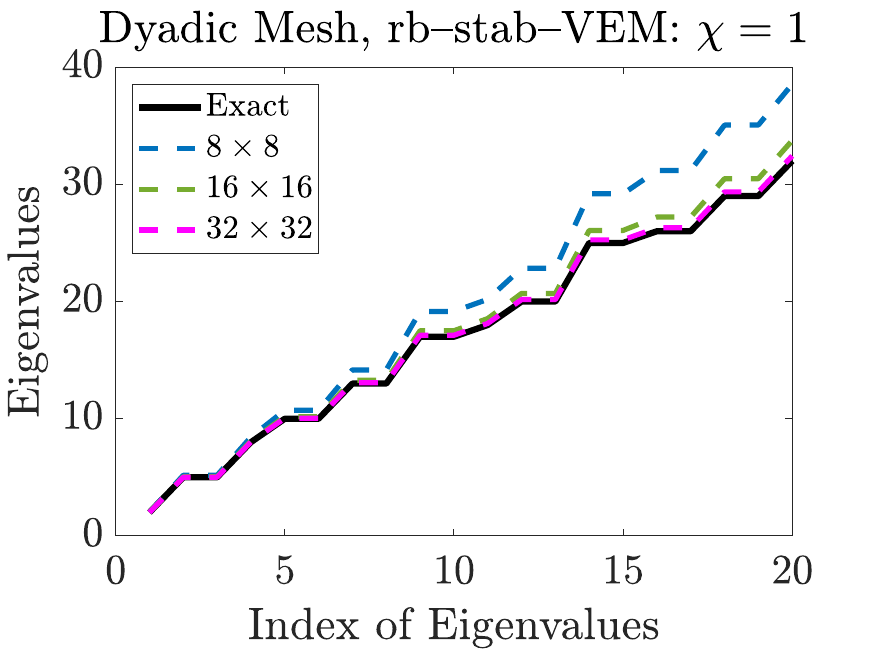}
		\includegraphics[width=0.325\linewidth]{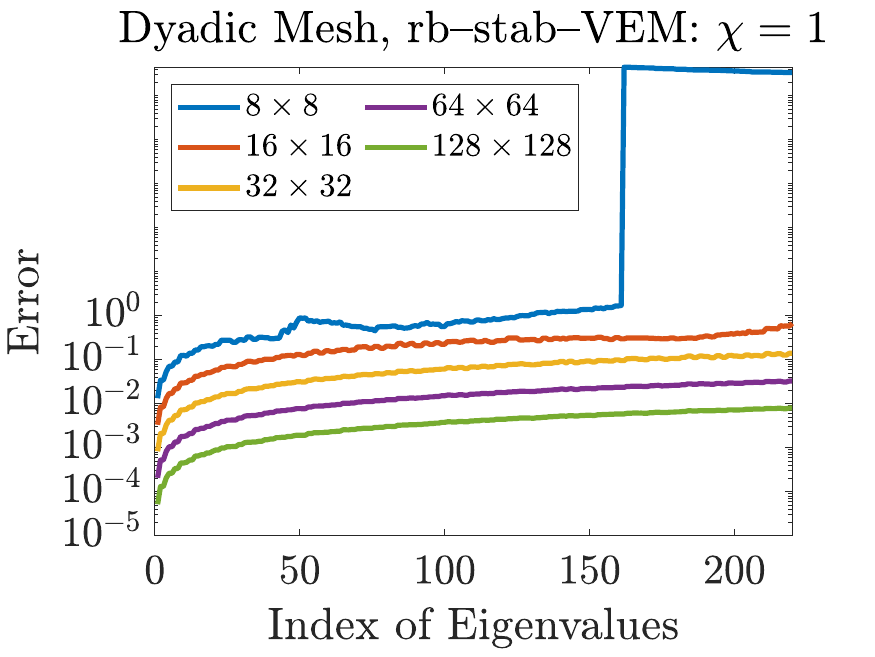}
		\includegraphics[width=0.325\linewidth]{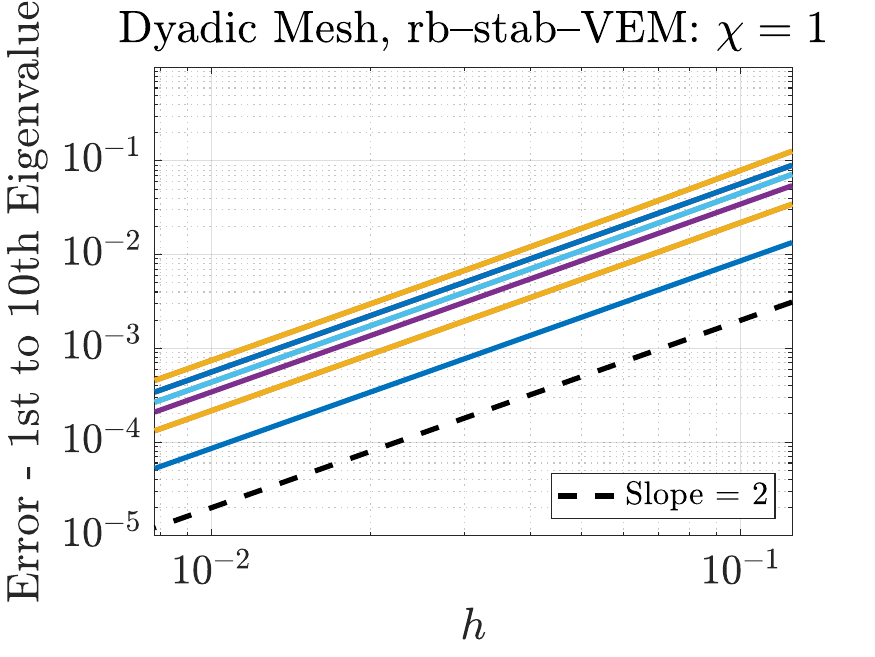}
		
		\
		
		\includegraphics[width=0.325\linewidth]{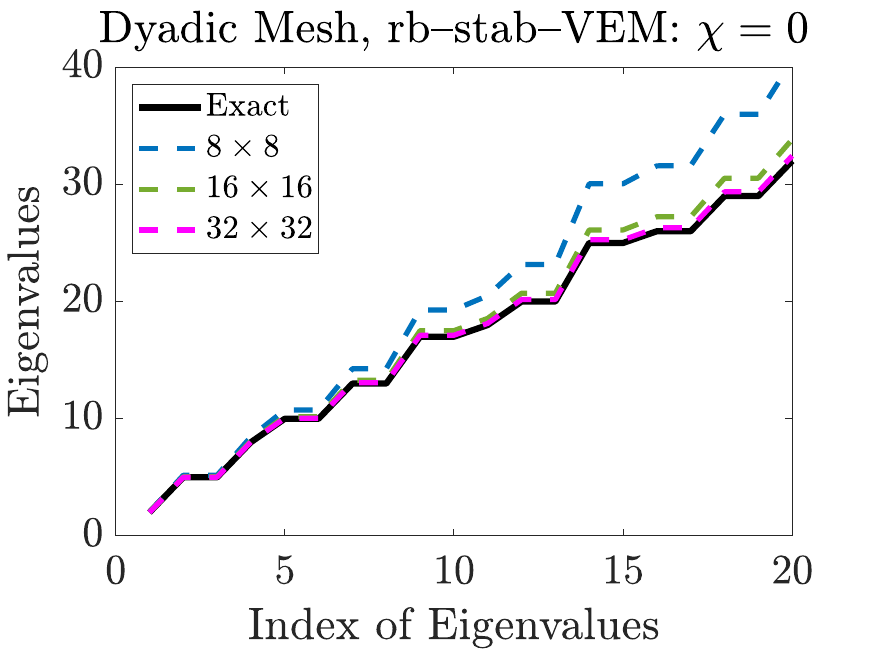}
		\includegraphics[width=0.325\linewidth]{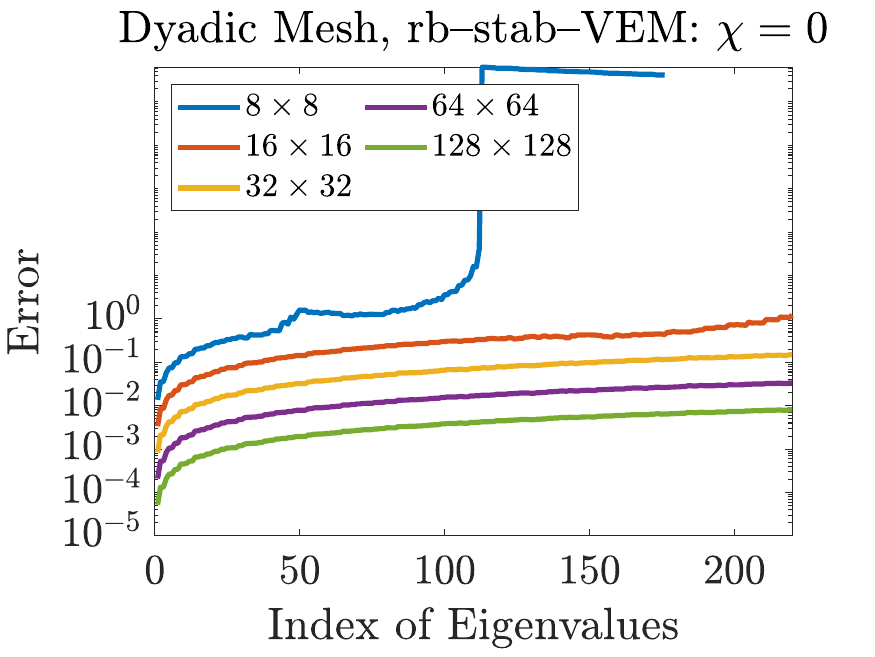}
		\includegraphics[width=0.325\linewidth]{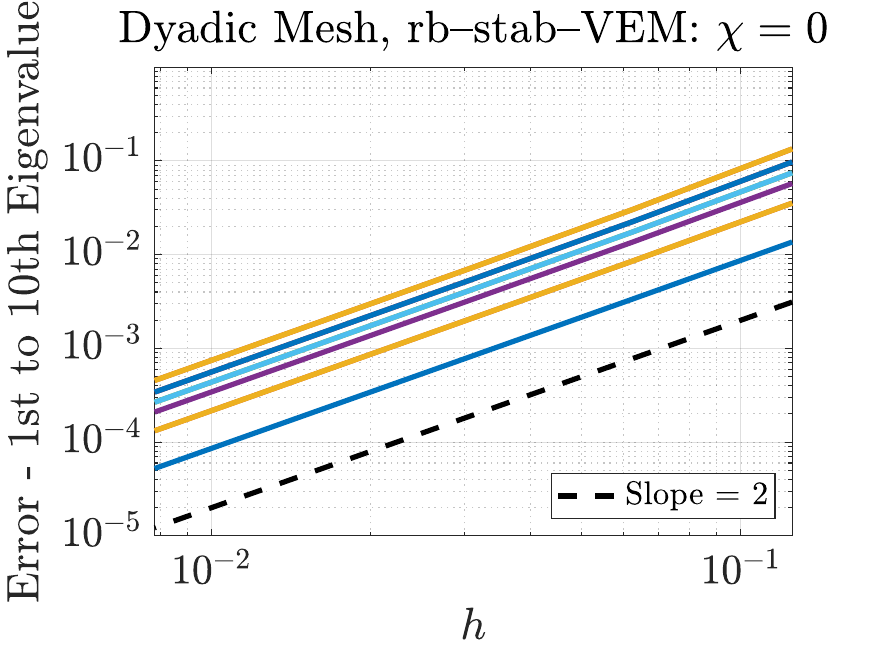}
		
		\caption{Comparison between VEM (first and second line) \textsf{rb}VEM (third line) and rb--stab--VEM (fourth and fifth line) on a sequence of dyadic meshes. First column: plots of the first 20 computed and exact eigenvalues. Second column: error plots for the first 220 eigenvalues with respect to their index. Third column: error plots for the first 10 eigenvalues when refining the mesh size.}
		\label{fig:dyadic}
	\end{figure}
	
	\subsection{Laplace eigenproblem on the L-shaped domain}
	
	We consider the benchmark test proposed in~\cite{dauge}. In this case, we solve the Laplace eigenproblem with homogeneous Neumann boundary conditions on the L-shaped domain $\Omega=(-1,1)^2 \setminus \Omega_0$, with $\Omega_0=(0,1)\times(-1,0)$, which is depicted in Figure~\ref{fig:ldomain}. Since an analytical expression for the exact eigenvalues is not available, our reference solutions are the first five eigenvalues computed on a very fine triangulation as given in~\cite{dauge}, i.e.
	\begin{equation*}
		\begin{aligned}
			&\lambda_1 = \text{0.147562182408\,E+01}\\
			&\lambda_2 = \text{0.353403136678\,E+01}\\
			&\lambda_3 = \text{0.986960440109\,E+01}\\
			&\lambda_4 = \text{0.986960440109\,E+01}\\
			&\lambda_5 = \text{0.113894793979\,E+02.}
		\end{aligned}
	\end{equation*}
	
	It is well-known that the eigenfunction $u_1$ associated with $\lambda_1$ is singular due to the re-entrant corner so that $u_1\in\Hr{1+s}{\Omega}$ for all $s<2/3$. Consequently, we expect a convergence rate of $4/3$. We point out that the eigenfunctions associated with $\lambda_2,\dots,\lambda_5$ are smooth, therefore the expected convergence rate of such eigenvalues is $2$.
	
	We report the convergence plots in Figure~\ref{fig:lconv}. More precisely, we analyze once again the stabilized VEM ($\alpha=\beta=1$), the partially stabilized VEM ($\alpha=1$, $\beta=0$), the novel \textsf{rb}VEM and the {rb--stab--VEM} ($\chi=0,1$). The four methods share the same convergence history; in particular, the \textsf{rb}VEM and the rb--stab--VEM are optimal even for singular eigenfunctions.
	
	\begin{figure}
		\centering
		\includegraphics[width=0.4\linewidth]{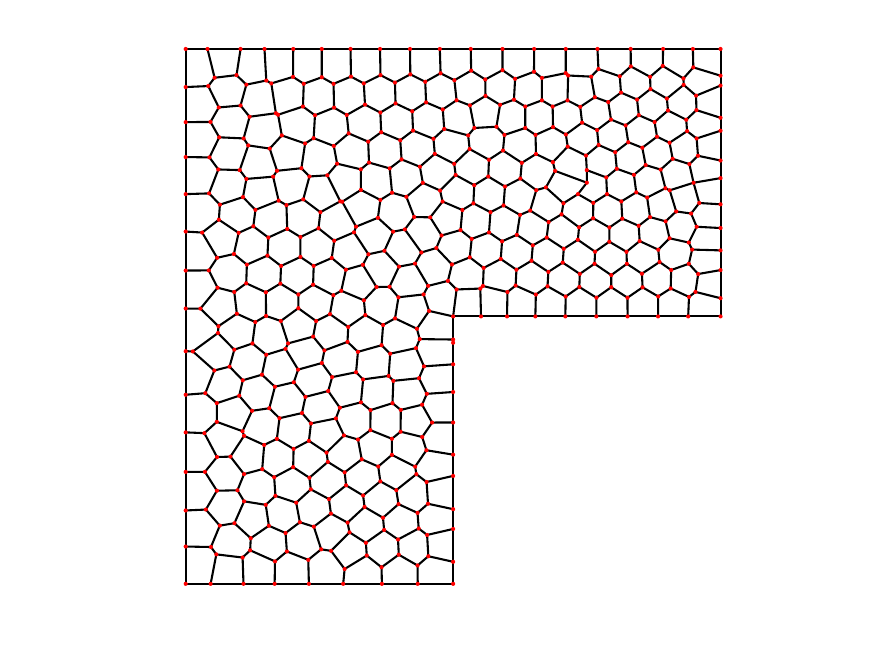}
		\caption{A Voronoi mesh for the L-shaped domain $\Omega=(-1,1)^2 \setminus \Omega_0$ with $\Omega_0=(0,1)\times(-1,0)$.}
		\label{fig:ldomain}
	\end{figure}
	
	\begin{figure}
		\centering
		
		\includegraphics[width=0.35\linewidth]{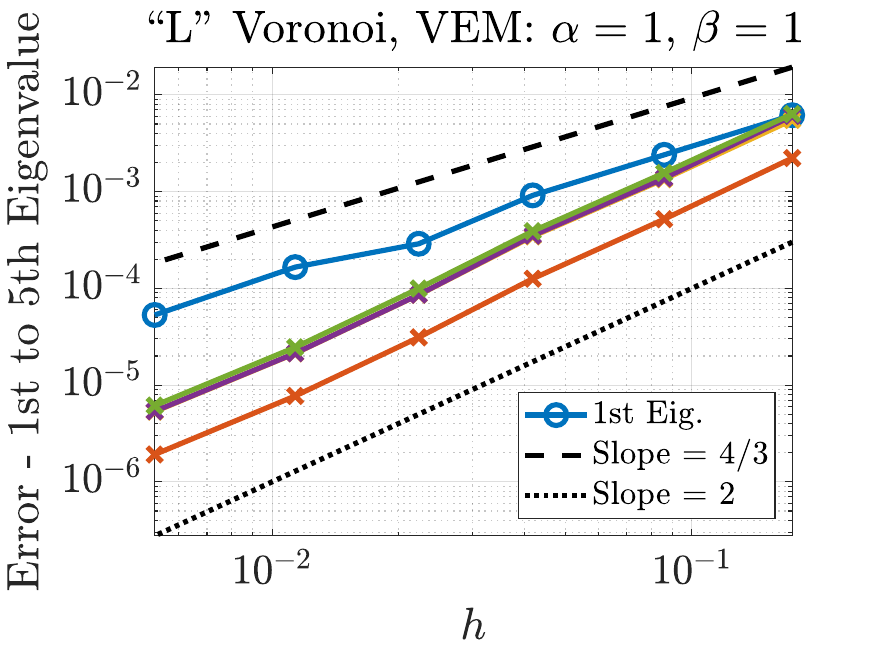}\qquad
		\includegraphics[width=0.35\linewidth]{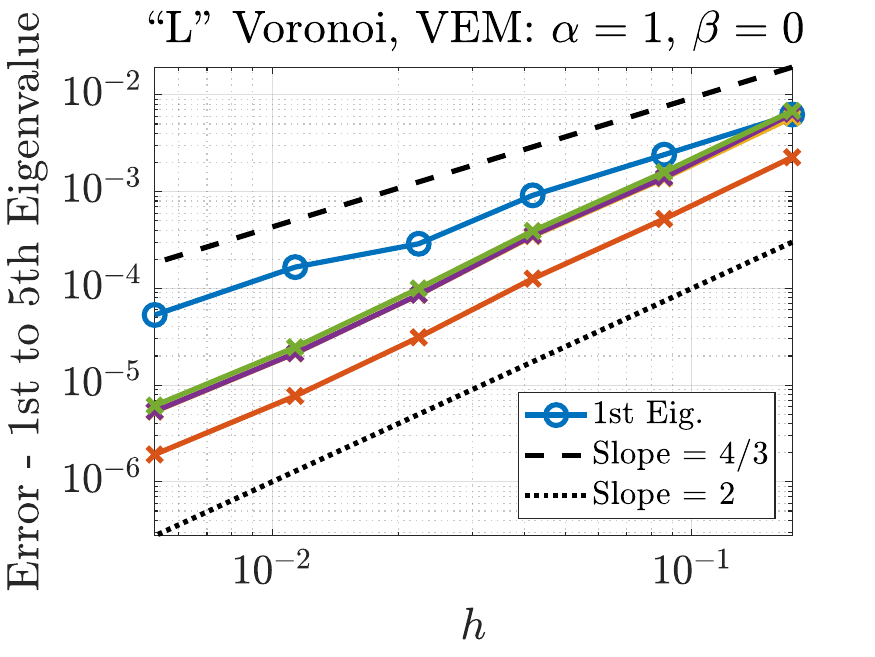}
		
		\
		
		\includegraphics[width=0.38\linewidth]{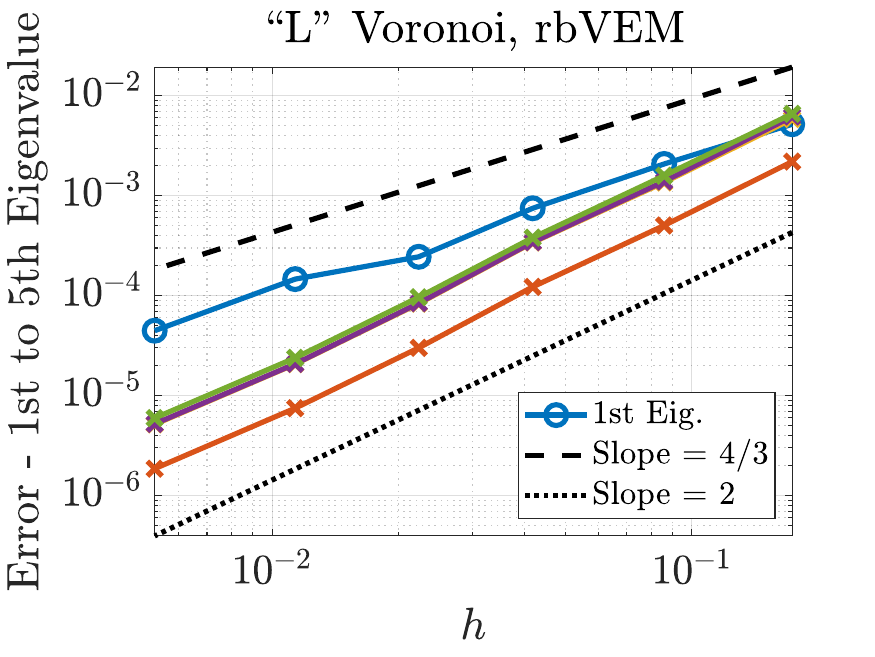}
		
		\
		
		\includegraphics[width=0.35\linewidth]{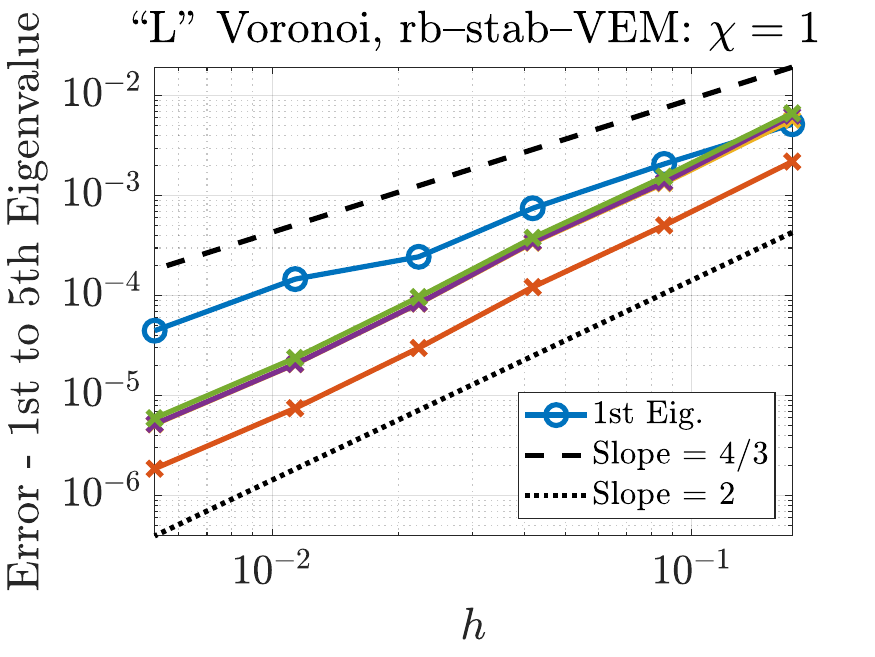}\qquad
		\includegraphics[width=0.35\linewidth]{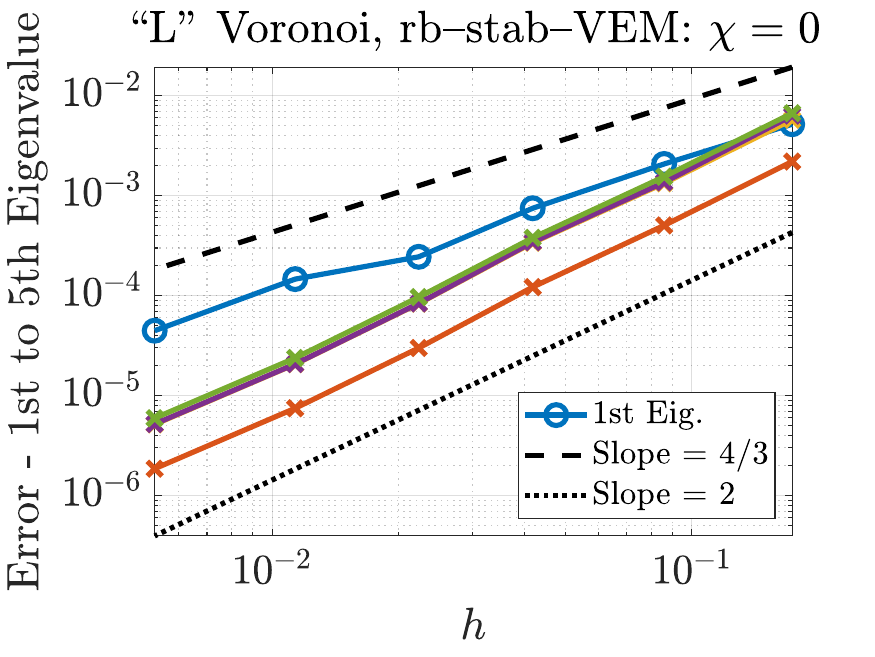}
		
		\caption{Convergence of eigenvalues with respect to $h$ on a sequence of Voronoi meshes for the L-shaped domain: VEM (first line), \textsf{rb}VEM (second line) and rb--stab--VEM (third line).}
		\label{fig:lconv}
		
	\end{figure}
	
	\subsection{Second order problem with piecewise constant coefficient}
	
	We consider another benchmark test from~\cite{dauge}: we seek for the eigenvalues of a second order problem with piecewise constant diffusivity \blue$\K$ \black and homogeneous Neumann boundary conditions. \blue The continuous bilinear form $a$ takes now the form
	$$
	a(u,v)=(\K\grad u,\grad v)_\Omega.
	$$
	The discrete bilinear form $a_h$ within the stabilized VEM setting is defined analogously to~\eqref{eq:discrete_bils}, see~\cite{beirao2016virtual}. On the other hand, all quantities are computed exactly in the case of~\textsf{rb}VEM.\black
	
	Let $\Omega=(-1,1)^2$ be decomposed into two regions $\Omega_1$ and $\Omega_\delta$ with diffusivity $\K=\eye$ and $\K=\delta\eye$, respectively, as depicted in Figure~\ref{fig:K_domain}. We discretize the domain by a family of Voronoi meshes and we consider four different values of $\delta$, namely $\delta=2,\,10,\,100,\,10^8$. We then compute the first two eigenvalues, whose reference solution computed on a very fine triangulation with high order finite elements is given by~\cite{dauge}
	\begin{equation*}
		\begin{aligned}
			&\delta=2: &&\qquad\delta=10: \\
			&\lambda_1=\text{0.3317548763415\,E+01},&&\qquad\lambda_1=\text{0.4533851871670\,E+01}, \\
			&\lambda_2=\text{0.3366324157260\,E+01}, &&\qquad\lambda_2=\text{0.6250332186603\,E+01}, \\[5pt]
			&\delta=100: &&\qquad\delta=10^8:\\
			&\lambda_1=0.\text{4893193324891\,E+01},&&\qquad\lambda_1=\text{0.4934802158785\,E+01}, \\
			&\lambda_2=0.\text{7206675422492\,E+01}, &&\qquad\lambda_2=\text{0.7225211232692\,E+01}.
		\end{aligned}
	\end{equation*}
	
	\begin{figure}[t]
		\centering
		\includegraphics[width=0.27\linewidth]{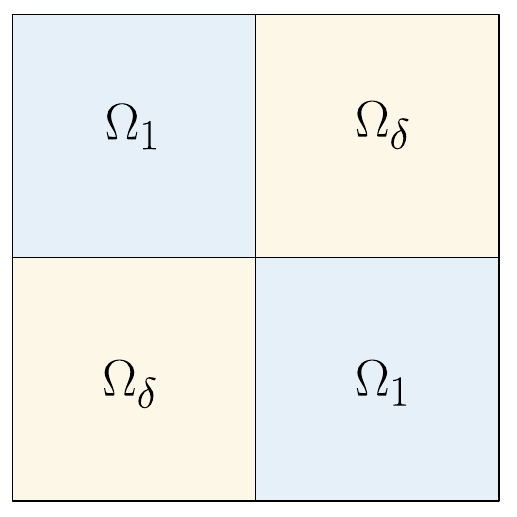}\qquad\qquad
		\includegraphics[width=0.27\linewidth]{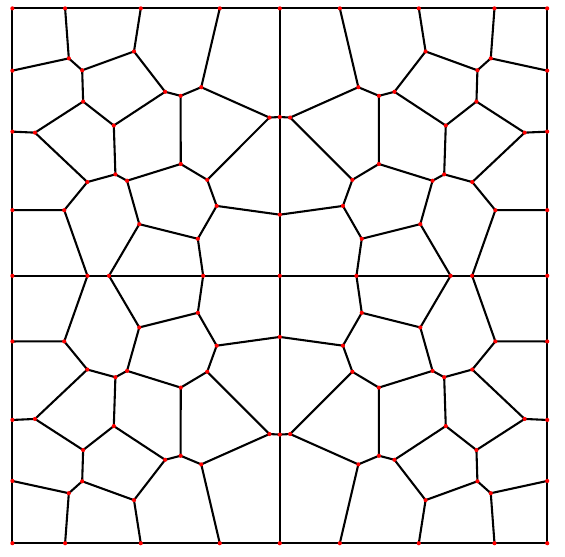}
		\caption{Subdivision of $\Omega$ into subdomains with different diffusivity and example of Voronoi mesh.}
		\label{fig:K_domain}
	\end{figure}
	
	Convergence plots of the first eigenvalue are reported in Figure~\ref{fig:K_1}, whereas those of the second eigenvalue are in Figure~\ref{fig:K_2}. More precisely, we compare \textsf{rb}VEM with the stabilized VEM ($\alpha=0.25,\,0.5,\,1$, $\beta=1$) and the {rb--stab--VEM} ($\chi=1$). We observe that \textsf{rb}VEM, {rb--stab--VEM} and VEM ($\alpha=\beta=1$) share the same optimal convergence history: all eigenvalues converge with rate two, except for the second one in correspondence of $\delta=10$, whose error decays with order 0.84, as already shown in~\cite{gardini2018}.  On the other hand, the stabilized VEM with $\alpha=0.25$ and $\alpha=0.5$ presents some oscillations, especially for the second eigenvalue in correspondence of $\delta=2,\,10$. 
	We finally point out that VEM and {rb--stab--VEM} have been also tested for $\beta=0$ and $\chi=0$, respectively: as in the previous test cases, the methods showed optimal rates of convergence and have not been plotted for the sake of readability.
	
	\begin{figure}[t]
		\centering
		\includegraphics[width=0.4\linewidth]{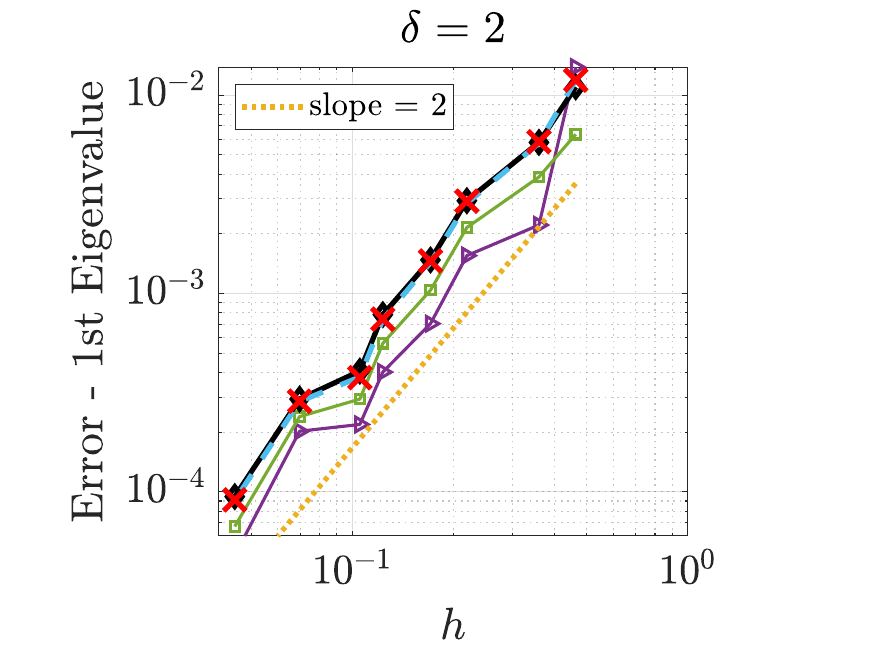}\quad
		\includegraphics[width=0.4\linewidth]{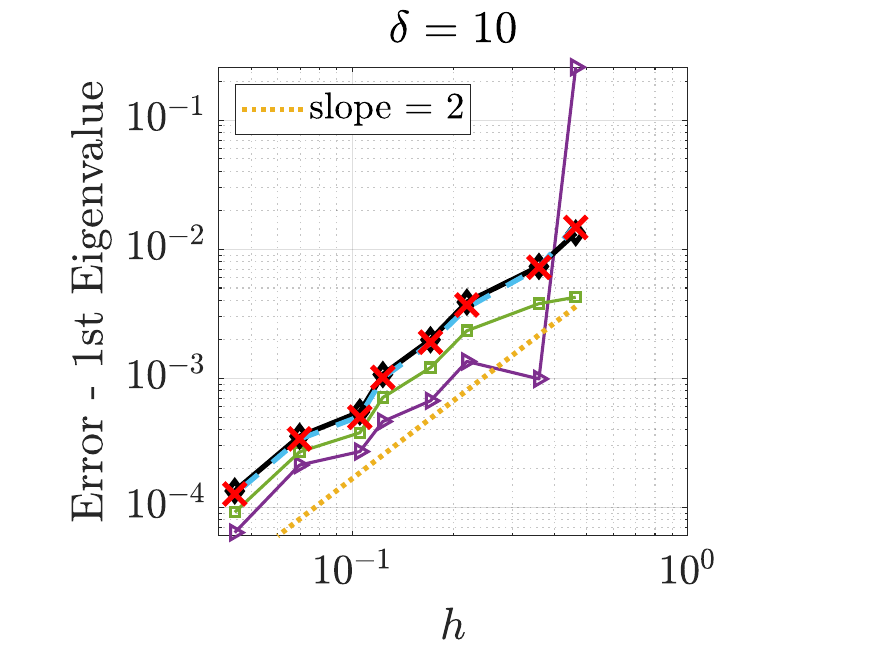}\quad
		
		\
		
		\includegraphics[width=0.4\linewidth]{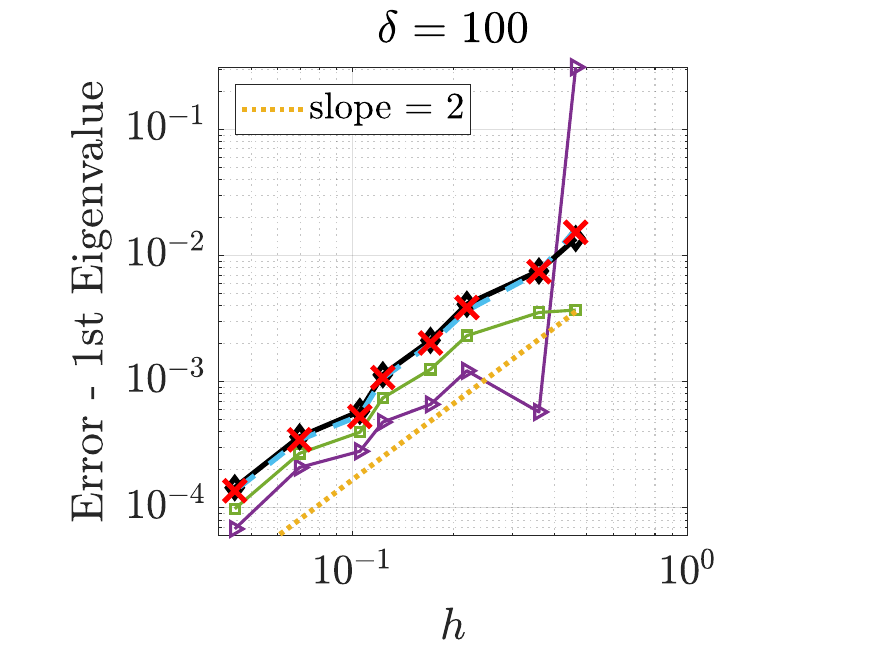}\quad
		\includegraphics[width=0.4\linewidth]{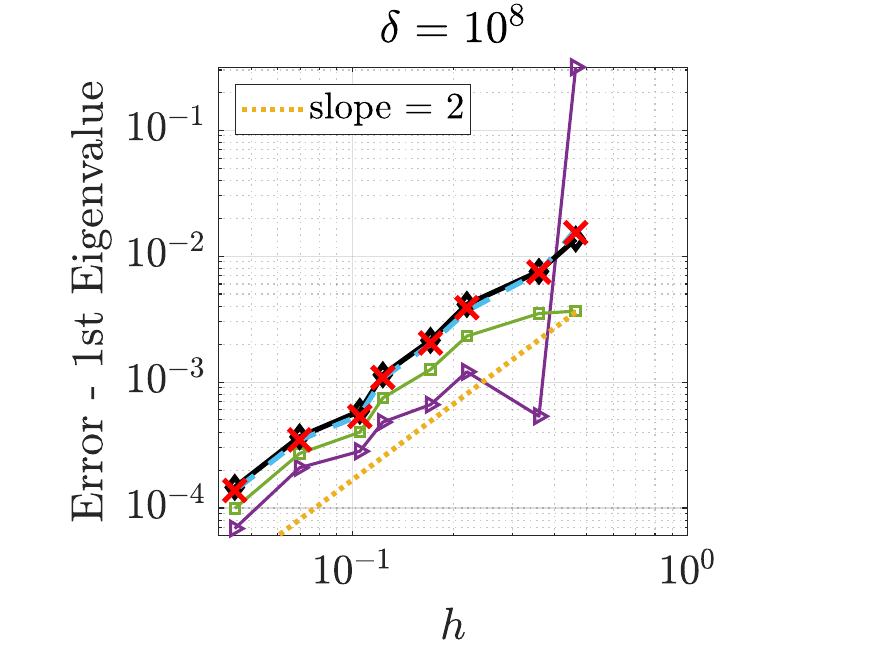}\quad
		
		\includegraphics[width=0.8\linewidth]{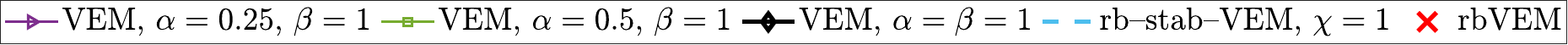}
		
		\caption{Convergence of the first eigenvalue for different values of diffusivity}
		\label{fig:K_1}
	\end{figure}
	
	\begin{figure}[t]
		\centering
		\includegraphics[width=0.4\linewidth]{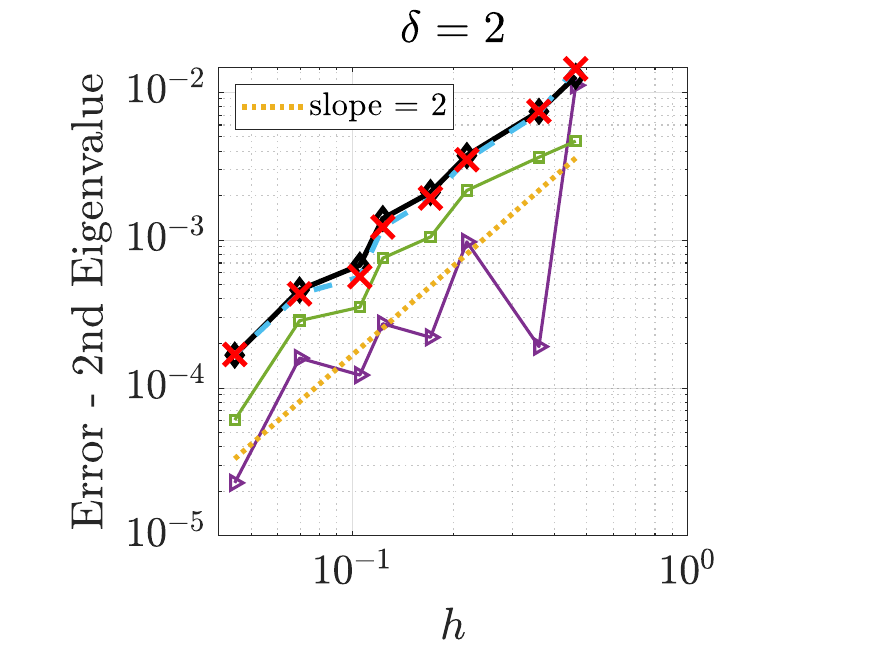}\quad
		\includegraphics[width=0.4\linewidth]{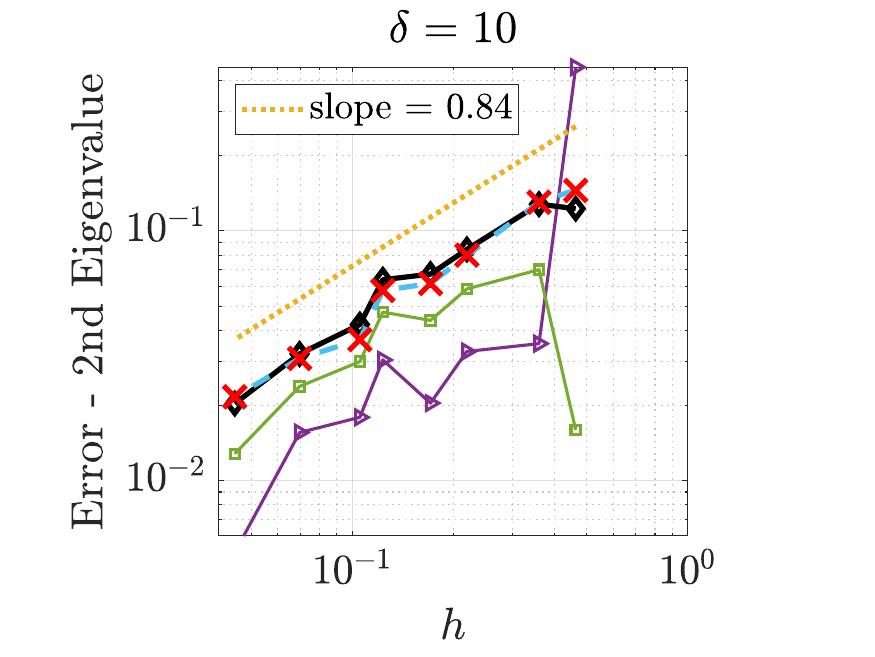}\quad
		
		\
		
		\includegraphics[width=0.4\linewidth]{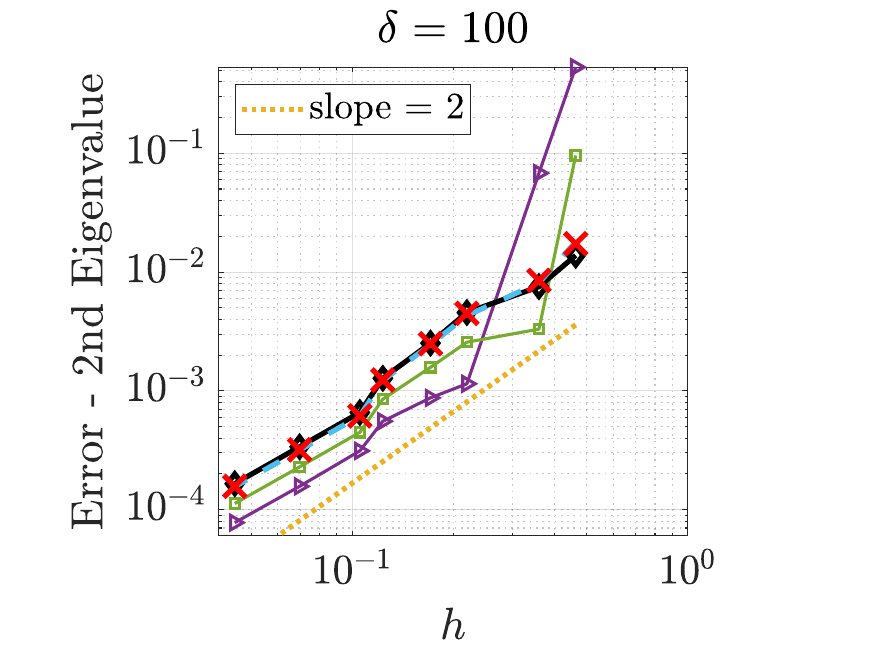}\quad
		\includegraphics[width=0.4\linewidth]{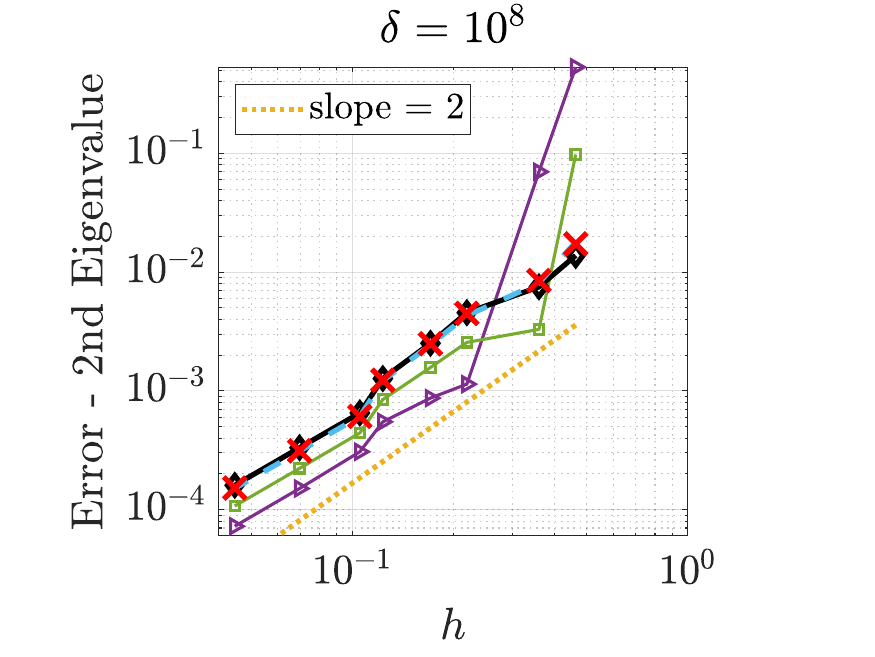}\quad
		
		\includegraphics[width=0.8\linewidth]{figures/K_coeff/K_legend}
		
		\caption{Convergence of the second eigenvalue for different values of diffusivity}
		\label{fig:K_2}
	\end{figure}
	
	\section{Conclusion}\label{sec:conclusions}
	
	We introduced the \rv lowest-order \black stabilization free \textsf{rb}VEM for the approximation of eigenvalue problems. The presence of stabilization parameters in standard VEM formulations may produce spurious eigenmodes. In order to avoid the parameters dependence, we employed the reduced basis method to efficiently solve the elemental equations defining the virtual basis functions. Thus, we were able to design a new fully conforming discrete space based on the original VEM.
	
	We proved that the \textsf{rb}VEM discretization of the Poisson problem is optimal in both $H^1-$ and $L^2-$ norms. Moreover, we established the correct spectral approximation of the Laplace eigenvalue problem, with optimal convergence rate of the eigenfuctions and the usual doubled rate for the eigenvalues. A wide range of numerical tests confirmed our theoretical findings.
	
	In our tests we also studied the behavior of alternative formulations \emph{\`a la VEM}, i.e. the rb--stab--VEM, where the reduced basis technique is employed to design the stability terms. This scheme performed equivalently to our \textsf{rb}VEM, which is covered by the theory. This behavior was expected since in virtual elements approximations the stability term is only required to scale as the actual non-computable term and not to fulfill good approximation properties.
	
	\rv 
	The high-order \textsf{rb}VEM for the source problem will be subject of our research in the near future, as well as its extension to eigenvalue problems.
	\black
	\section{Acknowledgments}	
	
	The authors are members of GNCS--INdAM Research Group and are grateful to Daniele Prada (IMATI--CNR Pavia) for providing the mesh files employed for the numerical tests.
	
	S. Bertoluzza  is partially supported by MUR–PRIN/PNRR Bando 2022 (grant P2022BH5CB).
	
	F. Gardini is partially supported by the National Recovery and Resilience Plan, Mission 4 Component 2 - Investment 1.4 - National Center for HPC, Big Data and Quantum Computing, spoke~6.
	
	\appendix
	
	\section{Implementation aspects of \textsf{rb}VEM}\label{app:appendix}
	
	\subsection{Computing the basis functions}\label{app:app1}
	
	In this appendix we recall the main features of the reduced basis approach we designed in~\cite{credali} for efficiently solving the parametric Problem~\ref{pro:refK}. The efficiency of the method relies on the \textit{affine decomposition} property.
	Indeed, from~\eqref{eq:def_aRB}, we can write
	\begin{equation}
		\aRB(\what,\vhat;\element) = \sum_{i=1}^N \aRB(\what,\vhat;\Tri_i)\quad\text{with}\quad\aRB(\what,\vhat;\Tri_i) = \left(|\det(\Bmatr^{-1}_{\element,i})|\, \Bmatr_{\element,i}^\top \Bmatr_{\element,i}\,\grad\widehat{w},\grad\widehat{v}\right)_{\That_i}.
	\end{equation}
	Since the matrix $|\det(\Bmatr^{-1}_{\element,i})|\, \Bmatr_{\element,i}^\top \Bmatr_{\element,i}$ is symmetric, we can rewrite it in terms of a basis $\sbase^1,\sbase^2,\sbase^3$ for the space of $2\times2$ symmetric matrices, i.e.
	\begin{equation}
		\det(\Bmatr^{-1}_{\element,i})|\,  \Bmatr_{\element,i}^\top \Bmatr_{\element,i} = \sum_{\nu=1}^3 \coeff \sbase^\nu.
	\end{equation}
	Hence we find the decomposition
	\begin{equation}\label{eq:aff_a}
		\aRB(\what,\vhat;\element) = \sum_{i=1}^N \sum_{\nu=1}^3 \coeff \aRBloc(\what,\vhat)\qquad\text{with}\quad\aRBloc(\what,\vhat)=\left(\sbase^\nu\grad\what,\grad\vhat\right)_{\That_i},
	\end{equation}	
	where only the coefficients $\coeff$ depend on the parameter, whereas the form $\aRBloc$ depends only on $\what$ and $\vhat$. Such decomposition is fundamental for the design of an efficient method. Indeed, once the reduced basis space is constructed, the building blocks $\aRBloc$ are computed and stored. Thus, for each instance of the parameter, the bilinear form $\aRB$ is easily computed exploiting the pre-computed objects. The same decomposition holds for the right hand side of Problem~\ref{pro:refK}, we have
	\begin{equation}\label{eq:aff_f}
		\aRB(\hTheta_j,\vhat;\element) = \sum_{i=1}^N \sum_{\nu=1}^3 \coeff \FRBloc(\vhat)\qquad\text{with}\qquad\FRBloc(\vhat)=\left(\sbase^\nu\grad\hTheta_j,\grad\vhat\right)_{\That_i}.
	\end{equation}
	
	In order to construct our reduced basis space, we \rv first \black solve \rv offline \black Problem~\ref{pro:refK} for a sufficiently large sample of polygons $\element_\ell\in\parameters$, for $\ell=1,\dots,L$. We thus obtain a collection of snapshots $\resKiell$ defined on the fine\footnote{
		\rv As usual in the reduced basis framework, $\meshdelta$ should be sufficiently fine to guarantee ``high fidelity'' of the approximate functions $\baserb{i}$. In our tests, we set $\delta=0.01$, as already done in~\cite{credali}.\black} mesh $\meshdelta$. The application of the Proper Orthogonal Decomposition \rv to such collection \black (see~\cite[Sect. 5.1]{credali} for more details) yields the construction of the reduced space, spanned by the functions $\rb_1^\ell,\dots,\rb_N^\ell$ for $\ell=1,\dots,\rv M\black$, which are linear combinations of $\resKiell$. At this stage, by exploiting the affine decomposition, we can pre-compute and store $\aRBloc$ and $\FRBloc$, so that we have
	\begin{equation}\label{eq:aff_dec_bricks}
		\Amatr_i^\nu(j,j^\prime,\ell,\ell^\prime) = \aRBloc(\rb_j^\ell,\rb_{j^\prime}^{\ell^\prime}), 
		\qquad
		\Fmatr_i^\nu(j,j^\prime,\ell) = \FRBloc(\rb_{j^\prime}^\ell) 
	\end{equation}
	for $\ell,\ell^\prime=1,\dots,M$, $j,j^\prime=1,\dots,N$, and $\nu=1,2,3$. We point out that $M\le L$ denotes the dimension (i.e. the accuracy) of the reduced basis space.
	
	For each new polygon $\element$ with $N$ vertices, we solve \rv online \black Problem~\ref{pro:refK} in the reduced space. We thus obtain $\resKirb$ in~\eqref{eq:rb_vem_basis_d} as the linear combination
	\begin{equation}
		\resKirb = \sum_{\ell=1}^{M} \rbcoeff_\ell \rb_j^\ell,
	\end{equation}
	where the coefficients $\rbcoeff_\ell$ are solution of the $M\times M$ linear system
	\begin{equation}
		\Amatr[\element]\rbcoeff=\Fmatr[\element],
	\end{equation}
	which is efficiently assembled by means of the affine decomposition bricks~\eqref{eq:aff_dec_bricks}, pre-computed only once at offline stage. Moreover, if $M$ is small, then the above system is also cheaply solved.
	We finally obtain the functions $\baserb{i} $ in~\eqref{eq:rb_vem_basis_2} as
	\begin{equation}\label{eq:rb_vem_basis}
		\baserb{i} = \baserbhat{i}\circ\map\qquad\text{where}\qquad\baserbhat{i} = \hTheta_i + \resKirb.
	\end{equation}
	
	\subsection{Computing the matrices in the rbVEM space}
	
	We now briefly describe how the idea underlying the reduced basis method can be exploited to cheaply compute the stiffness and mass matrices in the space $\Vhrb(\element) = \PunoK\oplus\Whrb(\element)$. We highlight that the polynomial contribution is computed as in standard virtual element framework, while the contribution in $\Whrb(\element)$ is computed using the affine decomposition. The entire procedure is summarized in Algorithm~\ref{alg:method}.
	
	The construction of $\Vhrb(\element)$ as the direct sum of two spaces implies that the stiffness matrix $\stiff$ can be decomposed as
	\begin{equation}
		\stiff=\mathsf{\Pi}^\nabla+\stiff^\mathrm{rb}.
	\end{equation} 
	The matrix $\mathsf{\Pi}^\nabla_{i,j}=\aK(\PiNabla\base_i,\PiNabla\base_j)$ is computed by projecting the basis functions $\base_1,\dots,\base_N$ of the standard VEM space $\vunoloc$ through the operator $\PiNabla$ defined in~\eqref{eq:pinabla}(see~\cite{beirao2014hitchhiker}). The term $\stiff^\mathrm{rb}$ is related to the virtual part of $\vunoloc$ which is not computable through the degrees of freedom. On the other hand, such contribution turns out to be fully computable in $\Whrb(\element)$.  Thus, we can write
	\begin{equation}
		\stiff^\mathrm{rb} = (\eye(N)-\mathsf{\Pi}^\nabla)^\top\, \mathsf{\Psi} \,(\eye(N)-\mathsf{\Pi}^\nabla),
	\end{equation}
	where $\eye(N)$ is the $N\times N$ identity matrix and $\mathsf{\Psi}_{i,j}=a^\element(\baserb{i},\baserb{j})$. \rv Notice that the space $\Whrb(\element)$ is not explicitly constructed as for standard FEM. Indeed, once $\mathsf{\Psi}$ is available, it is ``restricted'' to $\Whrb(\element)$ by using the degrees of freedom $\eye(N)-\mathsf{\Pi}^\nabla$ of virtual contribution.\black
	
	We now describe how to cheaply compute $a^\element(\baserb{i},\baserb{j})$. Let us denote by $\ahat$ the bilinear form on the reference element $\refelement$ obtained by change of variable in $\aK$, that is
	\begin{equation}
		\ahat(w,v)=\sum_{i=1}^N\left(|\det(\Bmatr^{-1}_{\element,i})|\, \Bmatr_{\element,i}^\top\,\K\, \Bmatr_{\element,i}\,\grad\what,\grad\vhat\right)_{\That_i},
	\end{equation}
	hence the following identity holds
	\begin{equation}
		a^\element(\baserb{i},\baserb{j}) = \ahat(\baserbhat{i},\baserbhat{j}).
	\end{equation}
	It is clear that all computations can be thus carried out in the reference element. Indeed, by using~\eqref{eq:rb_vem_basis}, we find the expansion
	\begin{equation*}\label{eq:sa_expansion}
		a^\element(\baserb{i},\baserb{j}) = \ahat(\hTheta_i,\hTheta_j)
		+ \sum_{\ell=1}^M \rbcoeff_\ell \ahat(\rb_i^\ell,\hTheta_j)
		+ \sum_{\ell=1}^M \rbcoeff_\ell \ahat(\hTheta_i,\rb_j^\ell)
		+ \sum_{\ell,\ell^\prime=1}^M \rbcoeff_\ell \ahat(\rb_i^{\ell^\prime},\rb_j^\ell),
	\end{equation*}
	where the bricks $\ahat(\rb_i^{\ell^\prime},\rb_j^\ell)$ are computed by affine decomposition as
	\begin{equation}
		\ahat(\rb_i^{\ell^\prime},\rb_j^\ell) = \sum_{i=1}^N \sum_{\nu=1}^4 \coeff \Amatr_i^\nu(j,j^\prime,\ell,\ell^\prime) 
	\end{equation}
	and an additional element $\sbase^4$ is added to $\sbase^1,\sbase^2,\sbase^3$ for dealing with the non-symmetric $2\times2$ matrix $\det(\Bmatr^{-1}_{\element,i})|\,  \Bmatr_{\element,i}^\top\,\K\, \Bmatr_{\element,i}$ as~\cite[Sect. 7]{credali}
	\begin{equation}
		\det(\Bmatr^{-1}_{\element,i})|\,  \Bmatr_{\element,i}^\top\,\K\, \Bmatr_{\element,i} = \sum_{\nu=1}^4 \coeff \sbase^\nu.
	\end{equation}
	
	The same reasoning applies for the construction of the mass matrix $\mass$ in $\Vhrb(\element)$. As already mentioned when constructing Problem~\ref{pb:rb_eig} in $\Vhrb(\element)$, here cross terms are nonzero, thus we have
	\begin{equation}
		\mass = \mass^\nabla+\mass^\mathrm{rb}+\mass^{\nabla,\mathrm{rb}}+(\mass^{\nabla,\mathrm{rb}})^\top.
	\end{equation}
	The term $\mass^\nabla$ is the mass matrix of the polynomial part, i.e. $\mass^\nabla_{i,j}=\bK(\PiNabla\base_i,\PiNabla\base_j)$. The second term is the mass matrix in the space $\Whrb(\element)$ and is assembled similarly to $\stiff^\mathrm{rb}$ as
	\begin{equation}
		\mass^\mathrm{rb} = (\eye(N)-\mathsf{\Pi}^\nabla)^\top\, \mathsf{\Phi} \,(\eye(N)-\mathsf{\Pi}^\nabla),
	\end{equation}
	where $\mathsf{\Phi}_{i,j}=\bK(\baserb{i},\baserb{j})$. As before, we define the bilinear form on the reference element
	\begin{equation}
		\bhat(w,v)=\sum_{i=1}^N \gamma_i[\element]\, \left(\what,\vhat\right)_{\That_i}\qquad\text{where}\qquad\gamma_i[\element]=|\det(\Bmatr^{-1}_{\element,i})|,
	\end{equation}
	so that the  following relation holds
	\begin{equation*}\label{eq:sb_expansion}
		\bK(\baserb{i},\baserb{j}) = \bhat(\hTheta_i,\hTheta_j)
		+ \sum_{\ell=1}^M \rbcoeff_\ell \bhat(\rb_i^\ell,\hTheta_j)
		+ \sum_{\ell=1}^M \rbcoeff_\ell \bhat(\hTheta_i,\rb_j^\ell)
		+ \sum_{\ell,\ell^\prime=1}^M \rbcoeff_\ell \bhat(\rb_i^{\ell^\prime},\rb_j^\ell).
	\end{equation*}
	Therefore, 
	having pre-computed the building blocks $\Mmatr_i(j,j^\prime,\ell,\ell^\prime)=(\rb_j^\ell,\rb_{j^\prime}^{\ell^\prime})_{\That_i}$ at offline stage, we easily construct
	\begin{equation}
		\bhat(\rb_j^\ell,\rb_{j^\prime}^{\ell^\prime}) = \sum_{i=1}^N \gamma_i[\element]\, \Mmatr_i(j,j^\prime,\ell,\ell^\prime).
	\end{equation}
	Finally, $\mass^{\nabla,\mathrm{rb}}$ is built as
	\begin{equation}
		\mass^{\nabla,\mathrm{rb}} = (\eye(N)-\mathsf{\Pi}^\nabla)^\top\, \mathsf{\tilde\Phi},
	\end{equation} 
	with $\mathsf{\tilde\Phi}_{i,j}=\bK(\baserb{i},\PiNabla\base_j)$.
	\begin{algorithm}
		\caption{Local assembler for the \textsf{rb}VEM}
		\label{alg:method}
		\begin{algorithmic}
			\State \textbf{Data: }$\element$, element of $\mesh$
			
			\
			
			\State\textbf{Compute projectors:}
			\State\quad Project basis of $\vunoloc$ onto $\PunoK$: $\PiNabla\base_j$, for $j=1,\dots,N$
			\State\quad Build $\mathsf{\Pi}^\nabla$ such that $\mathsf{\Pi}^\nabla_{i,j}=\aK(\PiNabla\base_i,\PiNabla\base_j)$
			\State\quad Compute dofs of virtual contribution:
			\State\qquad$\mathsf{R}=\eye(N)-\mathsf{\Pi}^\nabla$, i.e. $\Rmatr_{i,j}=\base_i(\vv_j)-\PiNabla\base_i(\vv_j)$ for $i,j=1,\dots,N$
			
			\
			
			\noindent\textbf{Reduced Basis Online Phase:}
			\State\quad Given $\element$, compute functions $\baserbhat{j}$ for $j=1,\dots,N$
			\State\quad by solving reduced system $\Amatr[\element]\rbcoeff=\Fmatr[\element]$
			
			\
			
			\State\noindent\textbf{Stiffness matrix:}
			\State\quad Compute affine decomposition coefficients $\coeff$ for $i=1,\dots,N$, $\nu=1,2,3,4$
			\State\quad Construct $\ahat(\rb_j^\ell,\rb_{j^\prime}^{\ell^\prime}) = \sum_{i=1}^N\sum_{\nu=1}^4 \coeff \Amatr_i^\nu(j,j^\prime,\ell,\ell^\prime)$ by affine decomposition
			\State\quad Build stiffness  $\mathsf{\Psi}$, i.e. $\mathsf{\Psi}_{i,j}=a^\element(\baserb{i},\baserb{j})$
			\State\quad Compute $\stiff^\mathrm{rb} = \Rmatr^\top\mathsf{\Psi}\Rmatr$
			\State\quad Compute stiffness matrix in \textsf{rb}VEM framework $\stiff=\mathsf{\Pi}^\nabla+\stiff^\mathrm{rb}$
			
			\
			
			\State\noindent\textbf{Mass matrix:}
			\State\quad Build $\mathsf{M}^\nabla$ such that $\mathsf{M}^\nabla_{i,j}=\bK(\PiNabla\base_i,\PiNabla\base_j)$
			\State\quad Compute affine decomposition coefficients $\gamma_i[\element]$ for $i=,1\dots,N$
			\State\quad Construct $\bhat(\rb_j^\ell,\rb_{j^\prime}^{\ell^\prime}) = \sum_{i=1}^N \gamma_i[\element]\, \Mmatr_i(j,j^\prime,\ell,\ell^\prime)$ by affine decomposition
			\State\quad Build mass $\mathsf{\Phi}$, i.e. $\mathsf{\Phi}_{i,j}=\bK(\baserb{i},\baserb{j})$
			\State\quad Compute $\mass^\mathrm{rb} = \Rmatr^\top\mathsf{\Phi}\Rmatr$
			\State\quad Build cross terms  $\mathsf{\tilde\Phi}$, i.e. $\mathsf{\tilde\Phi}_{i,j}=\bK(\baserb{i},\PiNabla\base_j)$
			\State\quad Compute $\mass^{\nabla,\mathrm{rb}} = \Rmatr^\top\mathsf{\tilde\Phi}$
			\State\quad Compute mass matrix in \textsf{rb}VEM framework $\mass=\m{M}^\nabla+\mass^\mathrm{rb}+\mass^{\nabla,\mathrm{rb}}+(\mass^{\nabla,\mathrm{rb}})^\top$
		\end{algorithmic}
	\end{algorithm}

	\bibliography{biblio.bib}
	\bibliographystyle{abbrv}
	
\end{document}

%% file: definitions.tex
\newcommand\eig{\lambda}
\newcommand\eigf{u}
\newcommand{\eigfrb}{u^\mathrm{rb}}

\newcommand{\eigvem}{\lambda_h}
\newcommand{\eigfvem}{u_h}

\newcommand\R{\mathbb{R}}
\newcommand\bOmega{\partial\Omega}

\newcommand{\Hunozero}{H_0^1(\Omega)}
\newcommand{\HunoK}{H^1(\element)}
\newcommand{\Hr}[2]{H^{#1}(#2)}
\newcommand\Ldue{L^2(\Omega)}

\newcommand{\grad}{\nabla}

\newcommand{\ds}{\mathrm{d}s}

\newcommand{\mesh}{\mathcal{T}_h}
\newcommand{\element}{E}
\newcommand{\edge}{\varepsilon}
\newcommand{\boK}{\partial\element}

\newcommand\PiNabla{\Pi^\nabla}
\newcommand\PiZero{\Pi^0}
\newcommand\PunoK{\mathbb{P}_1(\element)}
\newcommand\Punoe{\mathbb{P}_1(\edge)}

\newcommand\vunoloc{V_1(\element)}
\newcommand\Vh{V_h}
\newcommand\Vperp{V_\perp(\element)}
\newcommand{\aK}{a^\element}
\newcommand{\bK}{b^\element}
\newcommand{\ahK}{a_h^\element}
\newcommand{\bhK}{b_h^\element}
\newcommand{\ah}{a_h}
\newcommand{\bh}{b_h}
\newcommand{\staba}{S_a^\element}
\newcommand{\stabb}{S_b^\element}

\newcommand\m[1]{\mathsf{#1}}

\newcommand\neig{d_h}

\newcommand\Ecal{\mathcal{E}}
\newcommand\Fcal{\mathcal{F}}
\newcommand{\deltahat}{\hat{\delta}}

\newcommand\vertex{\mathbf{v}}
\newcommand\Nvert{N}
\newcommand{\para}{\sigma}
\newcommand{\parb}{\tau}

\newcommand{\vv}{\mathbf{v}}
\newcommand\base{e}
\newcommand\hbase{\widehat\base}
\newcommand\xK{\mathbf{x}_\element}
\newcommand{\radiusK}{\rho_\element}
\newcommand{\parameters}{\mathscr{P}}
\newcommand{\Ker}{Ker}
\newcommand{\refelement}{\widehat{\element}}
\newcommand\hv{\widehat{\vv}}
\newcommand\x{\mathbf{x}}
\newcommand\hx{\widehat{\x}}
\newcommand{\Tri}{T}
\newcommand{\That}{\widehat{\Tri}}

\newcommand{\map}{\mathscr{L}_\element} 
\newcommand{\Bmatr}{\mathsf{L}} 
\newcommand{\Amatr}{\mathsf{A}} 
\newcommand{\Fmatr}{\mathsf{F}} 
\newcommand{\Mmatr}{\m{B}} 
\newcommand\stiff{\m{K}}
\newcommand\mass{\m{M}}
\newcommand\eye{\m{Id}}
\newcommand\Cmatr{\m{C}} 
\newcommand\Rmatr{\m{R}} 

\newcommand\aRB{\mathscr{A}}

\newcommand\Vdelta{\mathcal{V}_\delta}
\newcommand\hTheta{\widehat{\Theta}}
\newcommand\meshdelta{\mathcal{T}_\delta}
\newcommand\resKi{\widehat d_{i,\delta}[\element]}
\newcommand\resKiell{\widehat d_{i,\delta}[\element_\ell]}
\newcommand\resKirb{\widehat d_{i,\delta}^\mathrm{rb}[\element]}

\newcommand\Huzh{H^1_0(\refelement)}
\newcommand\vhat{\widehat{v}}
\newcommand\what{\widehat{w}}
\newcommand\sbase{\mathcal{S}}
\newcommand\coeff{c_i^\nu[\element]}
\newcommand\aRBloc{\mathscr{A}^{i,\nu}}
\newcommand\FRBloc{\mathscr{F}^{i,\nu}_j}

\newcommand\rb{\widehat\xi}
\newcommand\rbcoeff{\mathsf{x}^{\element,j}}
\newcommand\ahat{\widehat{a}^\element}
\newcommand\bhat{\widehat{b}^\element}

\newcommand\baserb[1]{\base_{M,#1}^\mathrm{rb}[\element]}
\newcommand\baserbhat[1]{\hbase_{M,#1}^\mathrm{rb}[\element]}
\newcommand\Vhrb{V_h^\mathrm{rb}}
\newcommand\Whrb{W_h^\mathrm{rb}}

\newcommand\vRB{v_h^\mathrm{rb}}
\newcommand\eigRB{\eig_h^\mathrm{rb}}
\newcommand\eigfRB{\eigf_h^\mathrm{rb}}

\renewcommand{\u}{u}
\newcommand{\so}{T}
\newcommand{\uhrb}{\u_h^\mathrm{rb}}
\newcommand{\vhrb}{v_h^\mathrm{rb}}
\newcommand\interpolant{I_\element}
\newcommand{\uIrb}{\u_I^\mathrm{rb}}
\newcommand{\Ecalrb}{\Ecal^\mathrm{rb}}

\newcommand\udual{\eta}

\newcommand{\K}{\mathcal{K}}

\newcommand\baseVrb[1]{\tilde{e}_{#1}^\mathrm{rb}}